\documentclass[11pt]{amsart}
\usepackage[english]{babel}
\usepackage[utf8]{inputenc}

\usepackage[inner=1.1in, outer=1.1in, bottom = 2.9cm, top = 2.9cm]{geometry}

\usepackage{amsmath, amssymb, amsfonts, amsthm}
\usepackage{graphicx, overpic}
\usepackage{mathrsfs}
\usepackage{mathtools}
\usepackage[usenames,dvipsnames,svgnames,table]{xcolor}
\usepackage{enumerate}
\usepackage{enumitem}
\usepackage{shuffle} 
\usepackage{lipsum}
\usepackage{color}
\usepackage{eucal}
\usepackage{phoenician}
\usepackage{xkeyval}
\usepackage{tikz}
\usepackage{pst-node} 
\usepackage{tikz-cd} 
\usepackage[OT2,T1]{fontenc} 
\usepackage{thm-restate} 
\usepackage{float}
\usepackage{microtype}
\usepackage[strict]{changepage} 
\usepackage{esvect}
\usepackage{etoolbox}
\usepackage{mathrsfs}
\usepackage{wrapfig}
\usepackage{caption}
\usepackage{cite}
\usepackage{mathrsfs}
\DeclareMathAlphabet{\mathpzc}{OT1}{pzc}{m}{it}
\usepackage{bbding}
\usepackage{array}
\usepackage{makecell}
\usepackage{stmaryrd}
\usepackage{lmodern}
\usepackage{amsbsy}
\usepackage{esvect}
\usepackage{bbding}
\usepackage{thm-restate}
\usepackage[toc,page, titletoc,title]{appendix}
\graphicspath{ {figures/} }
\usepackage{url}
\usepackage[pagebackref]{hyperref} 

\makeatletter
\providecommand*{\twoheadrightarrowfill@}{%
  \arrowfill@\relbar\relbar\twoheadrightarrow
}
\providecommand*{\twoheadleftarrowfill@}{%
  \arrowfill@\twoheadleftarrow\relbar\relbar
}
\providecommand*{\xtwoheadrightarrow}[2][]{%
  \ext@arrow 0579\twoheadrightarrowfill@{#1}{#2}%
}
\providecommand*{\xtwoheadleftarrow}[2][]{%
  \ext@arrow 5097\twoheadleftarrowfill@{#1}{#2}%
}
\makeatother
\makeatletter
\newcommand*{\relrelbarsep}{.386ex}
\newcommand*{\relrelbar}{%
  \mathrel{%
    \mathpalette\@relrelbar\relrelbarsep
  }%
}
\newcommand*{\@relrelbar}[2]{%
  \raise#2\hbox to 0pt{$\m@th#1\relbar$\hss}%
  \lower#2\hbox{$\m@th#1\relbar$}%
}
\providecommand*{\rightrightarrowsfill@}{%
  \arrowfill@\relrelbar\relrelbar\rightrightarrows
}
\providecommand*{\leftleftarrowsfill@}{%
  \arrowfill@\leftleftarrows\relrelbar\relrelbar
}
\providecommand*{\xrightrightarrows}[2][]{%
  \ext@arrow 0359\rightrightarrowsfill@{#1}{#2}%
}
\providecommand*{\xleftleftarrows}[2][]{%
  \ext@arrow 3095\leftleftarrowsfill@{#1}{#2}%
}
\makeatother

\DeclareSymbolFont{cyrletters}{OT2}{wncyr}{m}{n}
\DeclareMathSymbol{\Sh}{\mathalpha}{cyrletters}{"58}

\usetikzlibrary{tikzmark,decorations.pathreplacing}
\usetikzlibrary{shapes,shadows,arrows}
\usetikzlibrary{decorations.markings}
\usetikzlibrary{positioning,calc}
\tikzset{near start abs/.style={xshift=1cm}}

\hypersetup{
	colorlinks=true,
	linkcolor=black,
	filecolor=black,      
	urlcolor=blue,
	citecolor= 	red,
}
\DeclareSymbolFont{extraup}{U}{zavm}{m}{n}
\DeclareMathSymbol{\varheart}{\mathalpha}{extraup}{86}
\DeclareMathSymbol{\vardiamond}{\mathalpha}{extraup}{87}

\newcommand{\bigslant}[2]{{\raisebox{.2em}{$#1$}\left/\raisebox{-.2em}{$#2$}\right.}}

\theoremstyle{definition}
\newtheorem{thm}{Theorem}[section]
\newtheorem{cor}{Corollary}[thm]

\newtheorem{prop}[thm]{Proposition}

\theoremstyle{definition}
\newtheorem{definition}{Definition}[section]
\newtheorem{ex}{Example}[section]

\newtheorem{remark}{Remark}[section]

\newcommand{\bR}{\mathbb{R}}

\newcommand{\cS}{\CMcal{S}}
\newcommand{\cA}{\CMcal{A}}
\newcommand{\cT}{\CMcal{T}}

\newcommand{\tH}{\mathtt{H}}

\newcommand{\Hom}{\operatorname{Hom}}

\newcommand{\Lie}{ \textbf{Lie} }

\newcommand{\Ass}{\textbf{Ass}}
\newcommand{\Lay}{\textbf{Lay}}

\newcommand{\oS}{\boldsymbol{\Gamma}}

\newcommand{\comma}{\text{,}}

\newcommand{\la}{\langle}
\newcommand{\ra}{\rangle}

\definecolor{Red}{rgb}{0.8,0,0.2}

\newcommand{\GG}[1]{}

\makeatletter
\def\@footnotecolor{red}
\define@key{Hyp}{footnotecolor}{%
 \HyColor@HyperrefColor{#1}\@footnotecolor%
}
\def\@footnotemark{%
    \leavevmode
    \ifhmode\edef\@x@sf{\the\spacefactor}\nobreak\fi
    \stepcounter{Hfootnote}%
    \global\let\Hy@saved@currentHref\@currentHref
    \hyper@makecurrent{Hfootnote}%
    \global\let\Hy@footnote@currentHref\@currentHref
    \global\let\@currentHref\Hy@saved@currentHref
    \hyper@linkstart{footnote}{\Hy@footnote@currentHref}%
    \@makefnmark
    \hyper@linkend
    \ifhmode\spacefactor\@x@sf\fi
    \relax
  }%
\makeatother

\hypersetup{footnotecolor=blue}

\makeatletter
\patchcmd{\@startsection}
  {\@afterindenttrue}
  {\@afterindentfalse}
  {}{}
\makeatother

\makeatletter
\newcommand*\Cdot{\mathpalette\Cdot@{.5}}
\newcommand*\Cdot@[2]{\mathbin{\vcenter{\hbox{\scalebox{#2}{$\m@th#1\bullet$}}}}}
\makeatother

\title[The Adjoint Braid Arrangement as a Lie Algebra]{The Adjoint Braid Arrangement as a Combinatorial\\ Lie Algebra via the Steinmann Relations}
\author{Zhengwei Liu}
\address[Zhengwei Liu]{Harvard University}
\email{zhengweiliu@fas.harvard.edu}
\author{William Norledge}
\address[William Norledge]{Pennsylvania State University and Harvard University}
\email{wxn39@psu.edu}
\author{Adrian Ocneanu}
\address[Adrian Ocneanu]{Pennsylvania State University and Harvard University}
\email{axo2@psu.edu}
\thanks{Supported by the Templeton Religion Trust grant TRT 0159 for the Mathematical Picture Language Project at Harvard University. The Harvard course of Adrian Ocneanu described work done by him at Penn State University, 1990-2017, partly supported by NSF grants DMS-9970677, DMS-0200809, DMS-0701589.}

\begin{document}

\renewcommand{\chapterautorefname}{Chapter}
\renewcommand{\sectionautorefname}{Section}
\renewcommand{\subsectionautorefname}{Section}

\begin{abstract}
We study a certain discrete differentiation of piecewise-constant functions on the adjoint of the braid hyperplane arrangement, defined by taking finite-differences across hyperplanes. In terms of Aguiar-Mahajan's Lie theory of hyperplane arrangements, we show that this structure is equivalent to the action of Lie elements on faces. We use layered binary trees to encode flags of adjoint arrangement faces, allowing for the representation of certain Lie elements by antisymmetrized layered binary forests. This is dual to the well-known use of (delayered) binary trees to represent Lie elements of the braid arrangement. The discrete derivative then induces an action of layered binary forests on piecewise-constant functions, which we call the forest derivative. Our main result states that forest derivatives of functions factorize as external products of functions precisely if one restricts to functions which satisfy the Steinmann relations, which are certain four-term linear relations appearing in the foundations of axiomatic quantum field theory. We also show that the forest derivative satisfies the Lie properties of antisymmetry the Jacobi identity. It follows from these Lie properties, and also crucially factorization, that functions which satisfy the Steinmann relations form a left comodule of the Lie cooperad, with the coaction given by the forest derivative. Dually, this endows the adjoint braid arrangement modulo the Steinmann relations with the structure of a Lie algebra internal to the category of vector species. This work is a first step towards describing new connections between Hopf theory in species and quantum field theory. 
\end{abstract}

\maketitle

\setcounter{tocdepth}{1} 
\tableofcontents

\section*{Introduction} \label{sec:Introduction}

The (essentialized) braid arrangement $\text{Br}[I]$ over a finite set $I$ is the hyperplane arrangement with ambient space 
\[
\text{T}[I]:= \bR^I/\bR =  \bigslant{   \{    \text{functions}\ I \to \bR  \} }{   \text{translations} }   
\]
and hyperplanes the fat diagonal
\[
\{ x_{i_1}-x_{i_2}=0  \}  \subset \text{T}[I]
\qquad
\text{for injective} \quad 
(i_1,i_2)\in I^2  
.\]
This is the reflection hyperplane arrangement of the type $A$ root system. In \cite{aguiar2010monoidal}, \cite{aguiar2013hopf}, Aguiar-Mahajan lay out rich connections between this hyperplane arrangement and combinatorial Hopf theory. They develop Hopf theory internal to the category of Joyal's vector species,\footnote{\ The theory of species was developed by André Joyal as the category-theoretic analog of analyzing combinatorial structures in terms of generating functions \cite{joyal1981theorie}, \cite{joyal1986foncteurs}. See also \cite{bergeron1998combinatorial}.} clarifying and generalizing the work of Barratt \cite{barratt1978twisted}, Joyal \cite{joyal1986foncteurs}, Schmitt \cite{Bill93}, Stover \cite{stover1993equivalence}, and others. Working internal to species provides a clear and unified perspective on combinatorial Hopf theory, with the plethora of graded Hopf algebras which appear in the literature being recovered via certain monoidal functors from vector species into graded vector spaces \cite[Chapter 15]{aguiar2010monoidal}. Aguiar-Mahajan emphasize the role which is played by the braid arrangement in providing consistent geometric interpretations of Hopf theory in Joyal's vector species. They went on to develop many aspects of the theory over generic real hyperplane arrangements \cite{aguiar2017topics}, \cite{aguiar2020bimonoids}.

The adjoint\footnote{\ in the sense of \cite[Section 1.9.2]{aguiar2017topics}} of the braid arrangement $\text{Br}^\vee[I]$ is the hyperplane arrangement with ambient space 
\[
\text{T}^\vee[I]:=  \Hom( \bR^I/\bR, \bR )\  =\   \text{sum-zero hyperplane of }\bR I 
\]
and hyperplanes
\[
\bigg\{ \sum_{i\in S} x_i=\sum_{i\in T} x_i=0\bigg  \}  \subset   \text{T}^\vee[I]
\qquad \text{for surjective} \quad 
  (S,T)\in 2^I
.\footnote{\ $(S,T)$ denotes the surjective function $I\to 2=\{1,2\}$ such that $S$ and $T$ are the preimages of $1$ and $2$ respectively}\]
This is the dual root space of type $A$, equipped with those hyperplanes which can be spanned by coroots. In this paper, we develop theory for the adjoint braid arrangement which gives rise to algebraic structure internal to Joyal's species. In particular, we obtain a new and highly structured geometric realization of one of the central Lie algebras appearing in Aguiar-Mahajan's work. This construction exposes connections between Hopf theory in species and quantum field theory, which are described in \cite{norledge2019hopf}, \cite{norledge2020species}.

We begin by studying the discrete differentiation of functions on faces of the adjoint braid arrangement, defined by taking \hbox{finite-differences} across hyperplanes. Originally, our motivation for understanding this structure came from studying permutohedral cones, see \cite{norledge2019hopf}. To describe the discrete derivative, let us first recall some basic aspects of the adjoint braid arrangement. For $P=(  S_1| \dots| S_k)$ a partition of $I$, the subspace
\[    
\text{T}^\vee[P] 
:=
\Big \{     (x_i)_{i\in I} \in \bR I   \ \Big | \  \sum_{i\in S_j} x_{i} =0 \ \text{ for all } \  1\leq j \leq k  \Big  \} 
\]
is a flat of the adjoint braid arrangement. We have a natural isomorphism
\[
 \text{T}^\vee[S_1]\times \dots \times \text{T}^\vee[S_k] \cong \text{T}^\vee[P]
.\]
We call such flats \emph{semisimple}, in reference to this factorization structure. The adjoint braid arrangement \emph{under} $\text{T}^\vee[P]$, denoted $\textsf{Br}^\vee[P]$, is the hyperplane arrangement with ambient space $\text{T}^\vee[P]$ and consists of those hyperplanes of $\text{T}^\vee[P]$ which are also flats of the adjoint braid arrangement. In general, $\textsf{Br}^\vee[P]$ almost has the structure of its corresponding product of arrangements, apart from the fact some extra `bad' hyperplanes are added, thus
\[
\text{Br}^\vee[S_1]   \times \dots \times \text{Br}^\vee[S_k]   \ncong \textsf{Br}^\vee[P]
.\]
The hyperplanes of $\textsf{Br}^\vee[P]$ which come from the product are exactly those which are semisimple, and the bad hyperplanes come from the fact that semisimple flats are not closed under intersection. See \autoref{subsec:The Restricted All-subset Arrangement} for more details. 

Let $\text{L}^\vee[I]$, resp. $\textsf{L}^\vee[P]$, denote the set of chambers of $\text{Br}^\vee[I]$, resp. $\textsf{Br}^\vee[P]$. For $\Bbbk$ a field of characteristic zero, let
\[
{\textbf{\text{L}}^\vee}^\ast[I]   :=   \big \{   \text{functions }   f:\text{L}^\vee[I]\to \Bbbk  \big \}
\qquad \text{and} \qquad
{\textbf{\textsf{L}}^\vee}^\ast[P]   :=   \big \{   \text{functions }   f:\textsf{L}^\vee[P]\to \Bbbk  \big \}
.\] 
The discrete differentiation of functions on chambers across the hyperplane of $(S,T)\in 2^I$ is formalized as the linear map
\[
\boldsymbol{\partial}_{[S,T]}  :{\textbf{\text{L}}^\vee}^\ast[I] \to {\textbf{\textsf{L}}^\vee}^\ast\big [S|T\big ]    
,\qquad
f \mapsto \boldsymbol{\partial}_{[S,T]} f 
\]
given by
\[       
\boldsymbol{\partial}_{[S,T]} f (\mathtt{X}):=  f(\mathtt{X}^{[S,T]})- f(\mathtt{X}^{[T,S]}),\qquad \mathtt{X}\in  \textsf{L}^\vee\big [S|T\big ]
\]
where $\mathtt{X}^{[S,T]}$ and $\mathtt{X}^{[T,S]}$ denote the two chambers of the adjoint braid arrangement which are adjacent to the codimension one face $\mathtt{X}$. We now pick a second hyperplane, corresponding to say $(U,V)\in 2^I$. Suppose that this second hyperplane is `non-overlapping'\footnote{\ this terminology comes from quantum field theory} in the sense that one of
\[     
S\cap U  \qquad \text{and}  \qquad      T\cap U 
\]
is empty (we elaborate on this non-overlapping property in \autoref{sec:$R$-Semisimplicity}). Then, differentiating for a second time as above but now taking finite-differences across the hyperplane of $(U,V)$, we obtain a linear map
\[
\boldsymbol{\partial}_{ (  [S_1,S_2]\, |\, [T] )}    :{\textbf{\textsf{L}}^\vee}^\ast\big [S|T\big ]     \to {\textbf{\textsf{L}}^\vee}^\ast\big [     S_1 | S_2 |T   \big ]    
\] 
where without loss of generality we have assumed $T\cap U=\emptyset$, and put $S_1=S\cap U$, $S_2=S\cap V$. In this way, non-overlapping iterations of our derivative are naturality encoded by layered binary forests, with each node of the forest corresponding to taking a derivative across a single hyperplane, see \autoref{sec:forestderiv}.  We call this the forest derivative. In full generality, the forest derivative is formalized as linear maps
\[
\boldsymbol{\partial}_{\mathcal{F}}  : {\textbf{\textsf{L}}^\vee}^\ast\big [P] \to {\textbf{\textsf{L}}^\vee}^\ast\big [Q]
\]
where $Q$ is a finer partition than $P$, and $\mathcal{F}$ is any appropriate layered binary forest, see \autoref{def:deriv}. 

Using the theory of Lie elements for generic real hyperplane arrangements \cite[Chapters 4 and 10]{aguiar2017topics}, the (dual of the) forest derivative can be obtained by representing certain Lie elements of the adjoint braid arrangement with layered binary forests, and then letting Lie elements act on faces. This perspective is described in \autoref{sec:lunes}. 

\begin{figure}[t]
	\centering
	\includegraphics[scale=0.75]{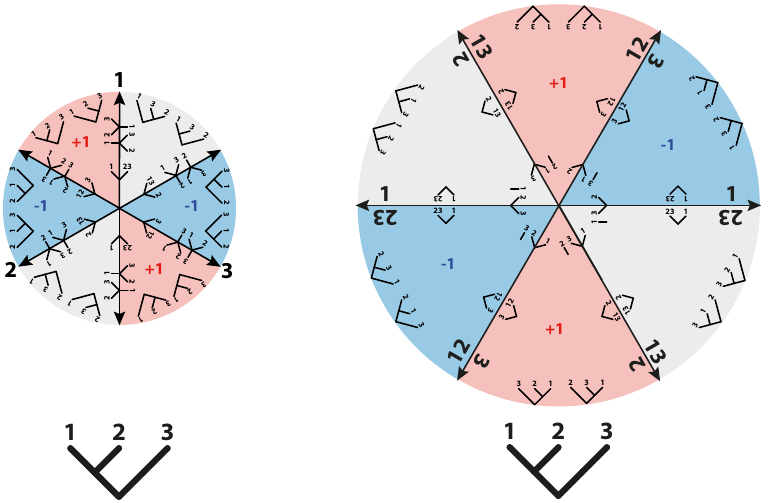}
	\caption{The braid and adjoint braid arrangements on $I=\{1,2,3\}$, \hbox{decorated} with a schematic for the action of the category of partitions on faces, which \hbox{interprets} layered binary trees as flags of faces. By antisymmetrizing trees, these actions allow us to use layered binary trees to represent Lie elements of both arrangements. Classically, this is only done for the braid arrangement (and with delayered trees) \cite{garsia90}, \cite{reutenauer}. The Lie elements represented by the tree $[[1,2],3]$ are shown.}
	\label{fig:action}
\end{figure} 

In \autoref{thm:lie}, we observe that this derivative has antisymmetry and satisfies the Jacobi identity, as interpreted on trees in the usual way, e.g. \cite{reutenauer}. The Jacobi identity is a consequence of the fact that the geometry of the adjoint braid arrangement imposes the following relations on trees, which can be seen in \autoref{fig:action},
\[   
[[1,2],3]=[[1,3],2] \qquad  \text{and} \qquad  [1,[2,3]]= [2,[1,3]]
.\] 
See \autoref{sec:lieprop} for details. Therefore we suspect a Lie (co)algebra is lurking somewhere. However, the objects we are dealing with are not single vector spaces; rather, they are whole families of vector spaces, one for each finite set $I$, or one for each partition $P$ of $I$, and they come equipped with an action of the symmetric group on $I$. In other words, they are \emph{vector species}. Therefore we should consider Lie (co)algebras internal to vector species.


Since a single derivative is with respect to a choice of $(S,T)\in 2^I$, this suggests we should consider positive Lie coalgebras constructed with respect to the so-called Cauchy monoidal product \cite[Definition 8.5]{aguiar2010monoidal}. Indeed, such an internal Lie coalgebra $\textbf{g}$ consists of a family of linear maps
\[
\textbf{g}[I] \to \textbf{g}[S] \otimes  \textbf{g}[T]
,\]
one for each choice of $(S,T)\in 2^I$. Now, since every hyperplane of the product \hbox{$\text{Br}^\vee[S_1]   \times \dots \times \text{Br}^\vee[S_k]$} is also a hyperplane of $\textsf{Br}^\vee[P]$, we have an external product of functions
\[  
\text{prod}_P: {\textbf{\text{L}}^\vee}^\ast[S_1]\otimes \dots \otimes {\textbf{\text{L}}^\vee}^\ast[S_k]  \hookrightarrow   {\textbf{\textsf{L}}^\vee}^\ast[P]  
\] 
with image those functions which do not change values across bad hyperplanes, see \autoref{def:prod}. We call a function $f\in {\textbf{\textsf{L}}^\vee}^\ast[P]$ \emph{semisimple} if it is contained in the image of this product. For example, in \autoref{fig:rootsolid} a semisimple function $f\in  {\textbf{\textsf{L}}^\vee}^\ast[12|34]$ must satisfy $f(\mathtt{X})=f(\mathtt{X}')$. 

We then obtain a would-be Lie cobracket by pulling back the derivative along the external product. Let $\boldsymbol{\Gamma}^\ast[P]$ denote the subspace of ${\textbf{\textsf{L}}^\vee}^\ast[P]$ consisting of \emph{semisimply differentiable} functions, which are functions whose non-overlapping iterated derivatives $\boldsymbol{\partial}_{\mathcal{F}} f$ are all semisimple. By restricting the derivative to semisimply differentiable functions on chambers $\boldsymbol{\Gamma}^\ast[I]:=\boldsymbol{\Gamma}^\ast[(I)]$, we obtain maps
\[
\boldsymbol{\partial}_{[S,T]} : \boldsymbol{\Gamma}^\ast[I] \to  \boldsymbol{\Gamma}^\ast\big [S|T\big ]
.\]
The fact that this restricted derivative does indeed land in $\boldsymbol{\Gamma}^\ast\big [S|T\big ]$ is a consequence of the fact that the derivative of an external product of functions factorizes as an external product of derivatives, see \autoref{prop:delayer}. In \autoref{cor1}, we see that the restricted external product is an isomorphism
\[
\text{prod}_P : \boldsymbol{\Gamma}^\ast [S_1] \otimes \dots \otimes  \boldsymbol{\Gamma}^\ast [S_k] \xrightarrow{\sim} \boldsymbol{\Gamma}^\ast [P]
.\] 
Thus, we obtain a Lie cobracket given by
\[
\partial_{[S,T]}   : \boldsymbol{\Gamma}^\ast[I] \to  \boldsymbol{\Gamma}^\ast\big [S]   \otimes   \boldsymbol{\Gamma}^\ast\big [T]  
,\qquad
\partial_{[S,T]} := \text{prod}_{(S|T)}^{-1}  \circ \boldsymbol{\partial}_{[S,T]}  
.\]
This gives $\boldsymbol{\Gamma}^\ast$ the structure of a Lie coalgebra internal to vector species. In this paper, we mainly consider the corresponding coaction of the Lie cooperad, which involves the iterated derivatives with respect to an arbitrary tree or forest, see \autoref{sec:coalg}. 

In \autoref{sec:The Steinmann Relations}, we observe that the property of a function's \emph{first} derivatives $\boldsymbol{\partial}_{[S,T]}f$ alone being semisimple is equivalent to the function satisfying certain four-term relations. These relations have previously appeared in the foundations of axiomatic quantum field theory, under the name Steinmann relations \cite{steinmann1960zusammenhang}, \cite[Equation 44]{steinmann1960}, \cite[p.827-828]{streater1975outline}. This connection is very interesting because the Steinmann relations were originally defined in a very different context to the current one, devoid of (explicit) \hbox{species-theoretic} considerations. In this paper, we arrive at the Steinmann relations by first observing Lie properties, and then restricting the ambient structure in an effort to obtain a Lie coalgebra. 

Our main result \autoref{main} states that if a function's first derivatives are semisimple (equivalently the function satisfies the Steinmann relations), then the function is semisimply differentiable. One can consider the functions which satisfy the Steinmann relations as differentiable functions, and semisimply differentiable functions as smooth functions. Then \autoref{main} is an analog of the result in complex analysis that a differentiable function is analytic. 

Therefore the Lie coalgebra $\boldsymbol{\Gamma}^\ast$ may be characterized as the subspace of functions in ${\textbf{\text{L}}^\vee}^\ast$ which satisfy the Steinmann relations. In other words, restricting to functions which satisfy the Steinmann relations is both necessary and sufficient in order for our discrete derivative to factorize and thus define a Lie cobracket. In the language of \cite{aguiar2020bimonoids}, the Steinmann relations are precisely what is required in order to move from a less structured \hbox{non-cartesian} setting to a more structured cartesian setting. In this way, our paper provides a pure mathematical `explanation' of the Steinmann relations.


\begin{figure}[t]
	\centering
	\includegraphics[scale=0.8]{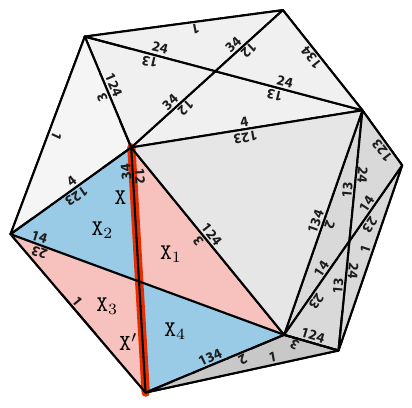}
	\caption{The intersection of the adjoint braid arrangement on $I=\{1,2,3,4\}$ with the coroot polytope of type $A_3$. Here, $\mathtt{X}$ and $\mathtt{X}'$ are codimension one faces contained in the flat corresponding to the partition $(12|34)$, and $\mathtt{X}_1$, $\mathtt{X}_2$, $\mathtt{X}_3$, $\mathtt{X}_4$ are chambers. The would-be Lie brackets of $\mathtt{X}$ and $\mathtt{X}'$ are $\boldsymbol{\partial}^\ast_{[12,34]} \mathtt{X}  =\mathtt{X}_1-\mathtt{X}_2$ and $\boldsymbol{\partial}^\ast_{[12,34]} \mathtt{X}'=\mathtt{X}_4-\mathtt{X}_3$ (see \autoref{sec:coalg}). However, so that the flat of $(12|34)$ factorizes like type $A_1 \times A_1$, the faces $\mathtt{X}$ and $\mathtt{X}'$ must be identified. Therefore we must have $\mathtt{X}_1-\mathtt{X}_2=\mathtt{X}_4-\mathtt{X}_3$, and so $\mathtt{X}_1-\mathtt{X}_2+\mathtt{X}_3-\mathtt{X}_4=0$, which is called a Steinmann relation.}
	\label{fig:rootsolid}
\end{figure}

We now discus some connections to Hopf theory in species. If we let $\oS[I]$ denote the quotient of $\textbf{L}^\vee[I]:=\Bbbk \text{L}^\vee [I]$ by the Steinmann relations, than dually we obtain a Lie algebra
\[
\partial^\ast_{[S,T]}:\oS[S] \otimes \oS[T] \to \oS[I]
,\qquad \partial^\ast_{[S,T]}:= (\partial_{[S,T]})^\ast
.\]
It is shown in \cite{norledge2019hopf} that $\oS$ is isomorphic to the free Lie algebra on the positive exponential species
\[
\oS \cong \CMcal{L}ie(\textbf{E}^\ast_+)=\textbf{Zie}
.\]
See \cite[Section 11.9]{aguiar2010monoidal}. We adopt the notation `\textbf{Zie}' from \cite{aguiar2017topics}. Thus, the universal enveloping algebra of $\oS$ is the cocommutative Hopf monoid of set compositions $\boldsymbol{\Sigma}$, which is defined in \cite[Section 11.1]{aguiar2013hopf}. The algebra $\boldsymbol{\Sigma}$ is fundamental to Hopf theory in species, and is closely related to the Tits algebra of the braid arrangement. The Tits product is recovered as the action of $\boldsymbol{\Sigma}$ on itself by Hopf powers \cite[Section 13]{aguiar2013hopf}. Moreover, it is shown in \cite{norledge2019hopf} that its dual Hopf algebra, the commutative Hopf monoid of set compositions $\boldsymbol{\Sigma}^\ast$, may be geometrically realized as an algebra of piecewise-constant functions on the adjoint braid arrangement which is spanned by characteristic functions of (generalized) permutohedral tangent cones. In particular, the multiplication of $\boldsymbol{\Sigma}^\ast$ is induced by embedding adjoint braid arrangement hyperplanes, and so the indecomposable quotient $\boldsymbol{\Sigma}^\ast\twoheadrightarrow \textbf{Zie}^\ast$ is obtained geometrically by restricting functions to open chambers. This explains why in this paper we are able to construct $\textbf{Zie}$ over the adjoint braid arrangement chambers, and also the appearance of signed quasi-shuffles in Ocneanu's relations for permutohedral cones \cite[\href{https://www.youtube.com/watch?v=mqHXVDwTGlI}{Lecture 36}]{oc17}. See also \cite{early2017canonical}.


The standard homomorphism $\boldsymbol{\Sigma}^\ast \twoheadrightarrow \textbf{L}^\ast$\footnote{\ $\cT (\varsigma)$ in \cite[Theorem 12.57]{aguiar2010monoidal}} is obtained geometrically on the adjoint braid arrangement by quotienting out functions which are constant in one or more coroot directions. For the usual geometric realization of $\boldsymbol{\Sigma}^\ast$ on the braid arrangement, the roles played by the geometry is vice versa; over the braid arrangement, $\boldsymbol{\Sigma}^\ast\twoheadrightarrow \textbf{Zie}^\ast$ manifests as quotienting out functions which are constant in one or more fundamental weight directions, and $\boldsymbol{\Sigma}^\ast \twoheadrightarrow \textbf{L}^\ast$ is the restriction of functions to open Weyl chambers. In particular, the braid arrangement realization of $\textbf{Zie}$, the so-called Zie elements \cite[Section 10.3]{aguiar2017topics}, is spread across faces of all codimensions, whereas the adjoint braid arrangement realization of $\textbf{Zie}$, i.e. $\oS$, exists purely on chambers. This switching of the correspondence between algebra and geometry occurs because on the braid arrangement, arrangements over flats factorize, whereas on the adjoint braid arrangement, it is the arrangements under flats which factorize.


Structures related to the Lie algebra $\oS$ and the adjoint braid arrangement have appeared in quantum field theory, see e.g. \cite{Huz1}, \cite{Ruelle}, \cite{epstein1976general}, \cite{epstein2016}, \cite{evans1992n}, \cite{evans1994being}. \hbox{Up-operators}\footnote{\ see \cite[Section 8.12]{aguiar2010monoidal}} on the species of the adjoint braid arrangement, called the Steinmann arrows, play a central role in the algebraic formalism developed in \cite{egs74}, \cite{epstein1976general} for the study of generalized retarded functions, see \autoref{arrow}. These up-operators may be extended to $\boldsymbol{\Sigma}$, and are used to obtain the (operator products of time-ordered products of) interacting field observables of a perturbative S-matrix scheme \cite[p. 259-261]{ep73roleofloc}. The primitive elements map $\oS\hookrightarrow \boldsymbol{\Sigma}$, which embeds the adjoint braid arrangement realization of $\textbf{Zie}$ into its universal enveloping algebra, appears in \cite[Equation 79, p. 260]{ep73roleofloc}, \cite[Equation 1, p.26]{epstein1976general}, \cite[Equations 35, 36]{epstein2016}. The map is given in terms of the $\tH$-basis of $\boldsymbol{\Sigma}$. We hope to further expose the connection between species and quantum field theory in future work. \medskip

\subsection*{Acknowledgments.} The authors are grateful to the Templeton Religion Trust, which supported this research with grant TRT 0159 for the Mathematical Picture Language Project at Harvard University. This made possible the visiting appointment of Adrian Ocneanu and the postdoctoral fellowship of William Norledge for the academic year 2017-2018. Adrian Ocneanu wants to thank Penn State for unwavering support during decades of work, partly presented for the first time during his visiting appointment. We also thank Nick Early for discussions related to permutohedral cones and the Steinmann relations \cite{early2017canonical}, \cite{early2017generalized}. After a literature search, Early discovered that the relations, which were conjectured by Ocneanu to characterize the span of characteristic functions of permutohedral cones, were known in axiomatic quantum field theory as the Steinmann relations. Zhengwei Liu would like to thank Arthur Jaffe for many helpful suggestions.\medskip 

This paper was inspired by a lecture course given by Adrian Ocneanu at Harvard University in the fall of 2017 \cite{oc17}. The lecture course is available on YouTube (see \href{https://www.youtube.com/watch?v=CJR-QXHP5GQ&list=PLdex6rFe0dfQRKLmcGC0rUrsswjHd01FS}{video playlist}). 

\section{The Adjoint Braid Arrangement}\label{subsec:The Restricted All-subset Arrangement}

\noindent \hyperlink{foo}{In} this section, we define the adjoint braid arrangement and describe some of its basic aspects. In particular, we emphasize the flats which are spanned by subsets of coroots. These flats are indexed by set partitions. We give a combinatorial description of both orthogonal projections of faces, and its linear dual, which is an external product of functions on faces. We show that by taking a certain quotient of faces, or dually by restricting to certain functions, projections and products become bijections. This gives a product structure on the adjoint braid arrangement under flats which are spanned by subsets of coroots. This product structure is required in order to obtain algebraic structure internal to Joyal's species.\footnote{\ see \cite[Preface, Differences]{aguiar2020bimonoids}}

\subsection{Flats of the Adjoint Braid Arrangement}


Let $I$ be a finite set with cardinality $n$. Let
\[      \bR^I:=\{  \text{functions}\     \lambda: I\to \bR    \} .     \]
For any subset $S\subseteq I$, let $\lambda_S\in \bR^I$ be given by
\[\lambda_S(i):=1 \quad \text{if} \quad  i\in S \qquad \text{and} \qquad \lambda_S(i):=0 \quad \text{if} \quad i\notin S.\] 
We have the free vector space on $I$, 
\[
\bR I:=\{  h=(h_i)_{i\in I} : h_i\in \bR   \}
.\] 
We have the perfect pairing
\[ \la-, - \ra:\bR I\times \bR^I\to \bR, \qquad     ( h, \lambda )\mapsto   \la h, \lambda  \ra := \sum_{i\in I}h_i \,  \lambda(i)    . \]
Let $\text{T}^\vee[I]$ denote the sum-zero hyperplane of $\bR I$,
\[\text{T}^\vee[I]  :=\big \{  h\in \bR I  :\la h, \lambda_I\ra =0    \big   \}. \]
For $i_1,i_2 \in I$ with $i_1\neq i_2$, the \emph{coroot} $h_{i_1 i_2}\in \text{T}^\vee[I]$ is defined by
\[\la h_{i_1 i_2},\lambda\ra:=\lambda(i_1)-\lambda(i_2).\]  
Then $\text{T}^\vee[I]$, together with the coroots $h_{i_1 i_2}$, forms the dual root system of type $A_{n-1}$.

\begin{definition}
A \emph{semisimple flat} is an $\bR$-linear subspace of $\text{T}^\vee[I]$ which can be spanned by coroots. 
\end{definition}

A \emph{partition} $P=( S_1|\dots| S_k )$ of $I$ is an (unordered) set of disjoint nonempty \emph{blocks} $S_j\subseteq I$ whose union is $I$. For partitions $P$ and $Q$ of $I$, we say that $Q$ is \emph{finer} than $P$ if every block of $Q$ is a subset of some block of $P$. We associate to each partition $P=( S_1| \dots| S_k)$ of $I$ the semisimple flat $\text{T}^\vee[P]$ which is given by
\[         
\text{T}^\vee[P]:=\big \{  h\in \bR I  : \lambda_{S_j}(h)=0\ \text{ for all }\ 1\leq j\leq k     \big \}  
.       \] 
Let $\sim_P$ denote the equivalence relation on $I$ which says that elements are equivalent if they are in the same block of $P$. We have
\[             \text{T}^\vee[P]=  \big  \la     h_{i_1 i_2}  : i_1 \sim_P i_2     \big \ra          .  \footnote{\ the angled brackets denote the $\bR$-linear span}   \]
Thus, $\text{T}^\vee[P]$ is indeed a semisimple flat. Conversely, for each semisimple flat $V\subseteq  \text{T}^\vee[I]$ there exists a unique partition $P$ of $I$ such that $V=\text{T}^\vee[P]$. The partition $P$ corresponds to the equivalence relation
\[i_1\sim i_2 \qquad \text{if} \quad    h_{i_1 i_2} \in V \ \  \text{or} \ \   i_1=i_2 .\] 
Therefore semisimple flats of $\text{T}^\vee[I]$ are in one-to-one correspondence with partitions of $I$. For partitions $P$ and $Q$ of $I$, $\text{T}^\vee[Q]$ is a subspace of $\text{T}^\vee[P]$ if and only if $Q$ is finer than $P$.  

Given a subset $S\subseteq I$, let $^{I\! } S$ denote the partition of $I$ which is the completion of $S$ with singletons, i.e. $^{I\! } S$ has $S$ as a block, and all its other blocks are singletons. A \emph{simple flat} is a semisimple flat of the form $\text{T}^\vee[^{I\! } S]$ for $|S|\geq 2$. We have a natural isomorphism 
\[
\text{T}^\vee[^{I\! } S]\cong \text{T}^\vee[S]
.\] 
For $P=( S_1| \dots| S_k)$, the semisimple flat $\text{T}^\vee[P]$ orthogonally factorizes into simple flats as follows,
\begin{equation} \label{eq} \tag{$\ast$}
  \text{T}^\vee[P]= \bigoplus_{|S_j|\geq 2}  \text{T}^\vee[^{I\! } S_j]\cong  \bigoplus_{|S_j|\geq 2}  \text{T}^\vee[S_j]   .
\end{equation}
The subspace $\text{T}^\vee[P]$ equipped with the coroots of $\text{T}^\vee[I]$ which are contained in $\text{T}^\vee[P]$ forms the dual root system of type 
\[     A_{ |S_1|-1} \times \dots \times   A_{ |S_k|-1}      .    \]
The factorization of $\text{T}^\vee[P]$ into simple flats is the decomposition of this dual root system into irreducible  dual root systems. 

A \emph{special hyperplane} is a semisimple flat which has codimension one in $\text{T}^\vee[I]$. An \emph{adjoint flat} is a subspace of $\text{T}^\vee[I]$ which is an intersection of a set of special hyperplanes of $\text{T}^\vee[I]$. If $P=( S_1| \dots| S_k)$ and $T_j=I- S_j$,\footnote{\ the minus denotes relative complement} then 
\[\text{T}^\vee[P]   = \bigcap^k_{j=1} \text{T}^\vee[S_j| T_j].\] 
Therefore every semisimple flat is an adjoint flat. However, the set of semisimple flats is not closed under intersection, and so there exist adjoint flats which are not semisimple.

\begin{definition}
The \emph{adjoint braid arrangement} $\mathrm{Br}^\vee[I]$ on $I$ is the hyperplane arrangement consisting of the special hyperplanes in $\text{T}^\vee[I]$. 
\end{definition}

The arrangement $\mathrm{Br}^\vee[I]$ is the adjoint of the more famous braid arrangement $\mathrm{Br}[I]$ (the adjoint of a real hyperplane arrangement is defined in \cite[Section 1.9.2]{aguiar2017topics}). We define $\mathrm{Br}[I]$ in \autoref{sec:lunes}. The adjoint braid arrangement also goes by the names the restricted all-subset arrangement \cite{kamiya2010ranking}, \cite{kamiya2012arrangements}, \cite{billera2012maximal}, \cite[Section 6.3.12]{aguiar2017topics}, the resonance arrangement \cite{MR2836109}, \cite{cavalierires}, \cite{billerabooleanprod}, \cite{gutekunst2019root}, and the root arrangement \cite{MR3917218}. Its spherical representation is called the Steinmann planet, or Steinmann sphere, in quantum field theory, e.g. \cite[Figure A.4]{epstein2016}, which may be identified with the boundary of the convex hull of coroots. The Steinmann sphere is the adjoint analog of the type $A$ Coxeter complex. 

\begin{definition}
The adjoint braid arrangement \emph{under} $\text{T}^\vee[P]$, denoted $\textsf{Br}^\vee[P]$, is the hyperplane arrangement with ambient space $\text{T}^\vee[P]$, and with hyperplanes the adjoint flats of $\text{T}^\vee[I]$ which are hyperplanes of $\text{T}^\vee[P]$.
\end{definition} 

We let $\textsf{Br}^\vee[^{I\! } S]$ denote the adjoint braid arrangement under $\text{T}^\vee[^{I\! } S]$. A natural isomorphism $\textsf{Br}^\vee[^{I\! } S]\cong \mathrm{Br}^\vee[S]$ is induced by the natural isomorphism of their underlying spaces. 

For $P=(S_1|\dots |S_k)$ a partition of $I$, the hyperplanes of $\textsf{Br}^\vee[P]$ which are semisimple flats of $\text{T}^\vee[I]$ are in natural bijection with the hyperplanes of the arrangements $\mathrm{Br}^\vee[^{I\! } S_j]$, $1\leq j\leq k$. However, if $P$ has at least two blocks which are not singletons, then $\textsf{Br}^\vee[P]$ will have additional hyperplanes which are not semisimple. Therefore (\ref{eq}) does not hold at the level of hyperplane arrangements. 

\subsection{$R$-Semisimplicity}\label{sec:$R$-Semisimplicity}

The hyperplanes of $\textsf{Br}^\vee[P]$ can be measured according to how far from being semisimple they are. In general, a hyperplane $V\subset \textsf{Br}^\vee[P]$ is obtained by choosing some proper and nonempty subset $E\subset I$ which is not a union of blocks of $P$, and then taking the subspace of $\text{T}^\vee[P]$ on which $\lambda_E=0$,
\[   V=\big \{   h\in \text{T}^\vee[P]:  \lambda_E(h)=0    \big  \}    . \] 
Let $P=(S_1|\dots |S_k)$. For any subset $E\subseteq I$, the \emph{reduction} $\text{Redn}_P(E)$ of $E$ with respect to $P$ is obtain by removing all the blocks of $P$ which are contained in $E$,
\[   \text{Redn}_P(E):=         E-  \bigcup_{E\cap S_j=S_j   } S_j.  \] 
Given subsets $E,S\subseteq I$, we write
\[   E=S \mod P  \qquad   \text{if} \qquad \text{Redn}_P (E)=\text{Redn}_P (S) .  \]
Let $E_1, E_2 \subseteq I$ be subsets with 
\[E_1,E_2\neq \emptyset \mod P.\] 
Then, the two hyperplanes of $\text{T}^\vee[P]$ on which $\lambda_{E_1}=0$ or $\lambda_{E_2}=0$ coincide if and only if
\[       E_1=E_2  \mod P    \qquad   \text{or} \qquad    E_1=I-E_2  \mod P   .        \]

\begin{definition}
Let $P=(S_1|\dots|S_k)$ be a partition of $I$, and let $R$ be a partition of $I$ such that $P$ is finer than $R$. Let $E\subseteq I$ be a subset such that 
\[E\neq \emptyset \mod P.\] 
Then $E$, and the hyperplane on which $\lambda_E=0$, are called \emph{$R$-semisimple} if $\text{Redn}_P (E)$ is contained within a single block of $R$. 
\end{definition}

Even though we just say $R$-semisimple, the definition depends on both $P$ and $R$. We think of the partition $R$ as the threshold of allowed `badness'\footnote{\ `badness' in the sense of not factorizing and obstructing species-theoretic algebraic structure} for the hyperplanes of $\textsf{Br}^\vee[P]$. A finer choice of $R$ corresponds to a stricter threshold, with the strictest choice being $R=P$. A hyperplane is $P$-semisimple if and only if it is a semisimple flat of $\text{T}^\vee[I]$. 

\begin{ex}
Let $I=\{1,\dots,9\}$, and let
\[       P=(12|34|56|78|9), \qquad R_1=(12|3456|789), \qquad    R_2=(123456|789)    .        \]
If $E_1=\{3,5,7,8\}$, then $\text{Redn}_P(E_1)=\{3,5 \}$, and $E_1$ is both $R_1$-semisimple and $R_2$-semisimple. In particular, $E_1$ is $R_2$-semisimple despite non-trivially intersecting two blocks of $R_2$. If $E_2=\{1,3,5\}$, then $\text{Redn}_P(E_2)=E_2$, and $E_2$ is $R_2$-semisimple but not $R_1$-semisimple.
\end{ex}

\subsection{Adjoint Faces}

We define an \emph{adjoint face} $\mathtt{X}$ over $I$ to be the characteristic function of a relatively open face\footnote{\ i.e. the relative interior of an intersection of closed halfspaces of the adjoint braid arrangement, with at least one associated halfspace chosen for each hyperplane, see \cite[Section 1.1.3]{aguiar2017topics}} of the adjoint braid arrangement $\mathrm{Br}^\vee[I]$. Let $[I;2]$ denote the set of ordered pairs $(S,T)$ such that 
\[
S\cup T=I,  \qquad S\cap T=\emptyset, \qquad S,T\neq \emptyset
.\] 
We can also think of $[I;2]$ as the set of surjective functions $I\to 2$, where $2$ is the ordinal with two elements. The \emph{signature} $\cS_h$ of a point $h\in \text{T}^\vee[I]$ is given by
\[         \cS_h := \big \{   (S,T)\in  [I;2]  :  \text{ $\lambda_S(h)\geq 0$}\big \} .  \]
If both $(S,T),(T,S)\in \cS_h$, then we say that $(S,T)$ is a \emph{symmetric} element of $\cS_h$. We put \[\cS_\mathtt{X}:=\cS_h\] 
for $h\in \text{T}^\vee[I]$ with $\mathtt{X}(h)=1$ (notice this is well-defined). Thus, $\cS_\mathtt{X}$ encodes on which sides of the special hyperplanes the relatively open face of $\mathtt{X}$ lies on. If $\cS=\cS_\mathtt{X}$, then we also denote $\mathtt{X}$ by $\mathtt{X}_\cS$. 

The following notation is useful; let the \emph{sign} $s_\mathtt{X}$ of an adjoint face $\mathtt{X}$ be the function on proper and nonempty subsets $S\subset I$ which is given by
\[   
 s_\mathtt{X}(S):= 
\begin{cases}
+  &  \quad \text{if } (S,T)\in  \cS_\mathtt{X}\ \text{and}\ (T,S)\notin  \cS_\mathtt{X}  \\
-  &  \quad \text{if }  (T,S)\in  \cS_\mathtt{X}\ \text{and}\ (S,T)\notin  \cS_\mathtt{X}  \\
0  &  \quad \text{if }   (S,T)\ \text{is a symmetric element}.  
\end{cases}
\]

\noindent We call the top dimensional adjoint faces \emph{adjoint chambers}, which are characterized by the property that $\cS_\mathtt{X}$ does not contain symmetric elements. Adjoint chambers correspond to the geometric cells of \cite{epstein1976general}, \cite{epstein2016}, and the corresponding sets $\cS_\mathtt{X}$ are called (or are equivalent to) cells. Note that cells also go by the name maximal unbalanced families. In \cite{billera2012maximal}, it is proved that the number of maximal unbalanced families, and so also the number of adjoint chambers, grows superexponentially with $n=|I|$. This is sequence \href{https://oeis.org/A034997}{A034997} in the OEIS.

The \emph{support} $\text{supp}(\mathtt{X})$ of an adjoint face $\mathtt{X}$ is the adjoint flat given by the following intersection of special hyperplanes,
\[     \text{supp}(\mathtt{X}):=\bigcap_{\substack{\text{symmetric elements}\\ (S,T)\in \cS_\mathtt{X}   }}   \text{T}^\vee[S|T]   .  \]
The support of $\mathtt{X}$ is equivalently the $\bR$-linear span of the relatively open face corresponding to $\mathtt{X}$.


Let $\Bbbk$ be a field of characteristic zero. For $P$ a partition of $I$, we have the vector space of $\Bbbk$-valued functions on the flat $\text{T}^\vee[P]$,
\[\Bbbk^{\text{T}^\vee[P]}:=\big \{ \text{functions}\ \text{T}^\vee[P] \to \Bbbk \big  \}.\] 
Let $\textbf{\textsf{L}}^\vee[P]\subset \Bbbk^{\text{T}^\vee[P]}$ denote the $\Bbbk$-linear span of the adjoint faces with support $\text{T}^\vee[P]$. The adjoint faces $\mathtt{X} \in \textbf{\textsf{L}}^\vee[P]$ are characterized by the property that $(S,T)$ is symmetric in $\cS_\mathtt{X}$ if and only if $S=\emptyset \mod P$. In particular, we have the $\Bbbk$-linear span of adjoint chambers, denoted
\[   
\textbf{L}^\vee[I]  :=\textbf{\textsf{L}}^\vee[(I)]   .\footnote{\ this notation is motivated by the fact that in the theory of Hopf monoids in species, the vector species of linear orders $\textbf{L}$ may be interpreted as formal $\Bbbk$-linear combinations of chambers of the braid arrangement}
\] 
Given a subset $S\subseteq I$, we have a natural isomorphism $\textbf{\textsf{L}}^\vee[^{I\! } S]\cong \textbf{L}^\vee[S]$.


 

\subsection{Projections of Adjoint Faces}\label{projj}

Let $R=( T_1|\dots| T_k  )$ be a partition of $I$, and let $P$ be a partition of $I$ which is finer than $R$. For $1\leq j \leq k$, let $P_j$ denote the partition of $T_j$ which is restriction of $P$ to $T_j$, and let $^{I\!} P_j$ denote the partition of $I$ which is the completion of $P_j$ with singletons.  

\begin{ex}
Let $I=\{ 1,2,3,4,5,6,7,8,9   \}$, $R=(1234|56|789)$, and $P=(12|34|56|7|89)$. Then 
\[
^{I\!} P_1=(12|34|5|6|7|8|9)
, \qquad 
^{I\!} P_2=(1|2|3|4|56|7|8|9)
, \qquad 
^{I\!}P_3=(1|2|3|4|5|6|7|89)   
.\] 
\end{ex}

We have the orthogonal factorization
\[      
\text{T}^\vee[P] =\bigoplus_{1\leq j \leq k}    \text{T}^\vee[^{I\!} P_j].\footnote{\ as $R$ becomes coarser for fixed $P$, the corresponding product of hyperplane arrangements will `miss' fewer hyperplanes; in particular, the finest case $R=P$ only sees hyperplanes which are semisimple}     
\]
For each adjoint face $\mathtt{X}\in \textbf{\textsf{L}}^\vee[P]$ and $1\leq j \leq k$, let $\mathtt{X}|_{T_j}$ denote the adjoint face in $\textbf{\textsf{L}}^\vee[^{I\!} P_j]$ which has sign
\[      s_{\mathtt{X}|_{T_j}}(S):=
\begin{cases}
s_{\mathtt{X}}(S\cap T_j)   &\quad\text{if $S\cap T_j\neq\emptyset$} \\
0  &\quad\text{if $S\cap T_j=\emptyset$.}
\end{cases}\]
To see that the adjoint face $\mathtt{X}|_{T_j}$ exists, notice that the orthogonal projection $h'$ of a point $h\in \text{T}^\vee[R]$ with $\mathtt{X}(h)=1$ onto $\text{T}^\vee[^{I\!} P_j]$ satisfies $\mathtt{X}|_{T_j}(h')=1$. 

\begin{definition}\label{def:proj}
Let $R=( T_1|\dots| T_k  )$ be a partition of $I$, and let $P$ be a partition of $I$ which is finer than $R$. Let $\mathtt{X}\in \textbf{\textsf{L}}^\vee[P]$ be an adjoint face. The \emph{projection} $\textbf{proj}_R(\mathtt{X})$ of $\mathtt{X}$ with respect to $R$ is given by
\[  
\textbf{proj}_R(\mathtt{X}):=   \mathtt{X}|_{T_1}\otimes  \dots \otimes \mathtt{X}|_{T_k}  \in  \textbf{\textsf{L}}^\vee[^{I\!} P_1] \otimes \dots \otimes \textbf{\textsf{L}}^\vee[^{I\!} P_k]
.\]
\end{definition}

\begin{prop}\label{prop:sur}
Let $R=( T_1|\dots| T_k  )$ be a partition of $I$, and let $P$ be a partition of $I$ which is finer than $R$. The map 
\[    
\textbf{proj}_R:    \textbf{\textsf{L}}^\vee[P]    \to  \textbf{\textsf{L}}^\vee[^{I\!} P_1]\otimes \dots \otimes \textbf{\textsf{L}}^\vee[^{I\!} P_k]
, \qquad  \ \   
\mathtt{X}\mapsto  \textbf{proj}_R(\mathtt{X})  
\]
is surjective.
\end{prop}

\begin{proof}
Let $\mathtt{X}_j\in \textbf{\textsf{L}}^\vee[^{I\!} P_j]$, for $1\leq j \leq k$, be any family of adjoint faces. The Minkowski sum of the relatively open faces corresponding to the $\mathtt{X}_j$ is open, since it is a cartesian product of open sets. Therefore it cannot be contained in the union of the hyperplanes of $\textsf{Br}^\vee[P]$. So we can find points $h_j$ with $\mathtt{X}_j(h_j)=1$ such that $h=\sum_j h_j$ satisfies $\mathtt{X}(h)=1$ for some $\mathtt{X}\in\textbf{\textsf{L}}^\vee[P]$. We have $ \mathtt{X}|_{T_j}=  \mathtt{X}_j$, and so
\[ \textbf{proj}_R (\mathtt{X} )   =  \mathtt{X}_1\otimes \dots \otimes \mathtt{X}_k .   \qedhere\] 
\end{proof}

If $R=P$, then we can make the identification $\textbf{\textsf{L}}^\vee[^{I\! } S_j]= \textbf{L}^\vee[S_j]$ to define
\[
\text{proj}_P:    \textbf{\textsf{L}}^\vee[P] \to  \textbf{L}^\vee[S_1]\otimes \dots \otimes \textbf{L}^\vee[S_k], \qquad  \ \   \mathtt{X} \mapsto  \textbf{proj}_P(\mathtt{X})      
.      \]

\subsection{Products of Functionals on Adjoint Faces}\label{sec:Functions on Shards}


We have the dual space 
\[   
 {\textbf{\textsf{L}}^\vee}^\ast[P]:=   \Hom(\textbf{\textsf{L}}^\vee[P], \Bbbk) = \big \{ \text{$\Bbbk$-linear functions $f:\textbf{\textsf{L}}^\vee[P]\to \Bbbk$}    \big \}
.\] 
Since adjoint faces form a basis of $\textbf{\textsf{L}}^\vee[P]$, the elements of ${\textbf{\textsf{L}}^\vee}^\ast[P]$ are naturally $\Bbbk$-valued functions on the adjoint faces $\mathtt{X}\in \textbf{\textsf{L}}^\vee[P]$. Since adjoint faces are already functions on $\text{T}^\vee[I]$, from now on we refer to elements $f\in {\textbf{\textsf{L}}^\vee}^\ast[P]$ as \emph{functionals}. 

We call a functional on adjoint faces $f\in {\textbf{\textsf{L}}^\vee}^\ast[P]$ \emph{simple} if it is supported by a simple flat, i.e. if the blocks of $P$ are all singletons, apart from one subset $S\subseteq I$ with $|S|\geq 2$, so that $f\in {\textbf{\textsf{L}}^\vee}^\ast[^{I\! } S]$. In particular, functionals on adjoint chambers are simple. 

\begin{definition}\label{def:prod}
Let $R=( T_1|\dots| T_k  )$ be a partition of $I$, let $P$ be a partition of $I$ which is finer than $R$, and let $^{I\!} P_j$ continue to denote the partition of $I$ which is the completion with singletons of the restriction of $P$ to $T_j$. Let $\textbf{\textit{f}}$ be a tensor product of functionals,
\[    
\textbf{\textit{f}}=f_1\otimes \dots \otimes f_k \in    {\textbf{\textsf{L}}^\vee}^\ast[^{I\!} P_1] \otimes \dots \otimes  {\textbf{\textsf{L}}^\vee}^\ast[^{I\!} P_k]    
. \]
The \emph{external product} $\textbf{prod}_R(\textbf{\textit{f}})$ of $\textbf{\textit{f}}$ with respect to $R$ is the functional in ${\textbf{\textsf{L}}^\vee}^\ast[P]$ whose value taken on each adjoint face $\mathtt{X}\in \textbf{\textsf{L}}^\vee[P]$ is the product of the values taken by the $f_j$ on the projections of $\mathtt{X}$ onto $\text{T}^\vee[^{I\!} P_j]$, thus
\[       
\textbf{prod}_R(\textbf{\textit{f}})(\mathtt{X}) := f_{1} (\mathtt{X}|_{T_1} )\dots f_{k} (\mathtt{X}|_{T_k} ), \qquad  \mathtt{X}\in \textbf{\textsf{L}}^\vee[P]         
.\]
\end{definition}

Notice that $\textbf{prod}_R$ is just the linear dual map of $\textbf{proj}_R$. Therefore, by \autoref{prop:sur}, we obtain an injective linear map
\[
\textbf{prod}_R: {\textbf{\textsf{L}}^\vee}^\ast[^{I\! }P_1]\otimes \dots \otimes {\textbf{\textsf{L}}^\vee}^\ast[^{I\! }P_k]\hookrightarrow {\textbf{\textsf{L}}^\vee}^\ast[P], \qquad  \ \    \textbf{\textit{f}} \mapsto \textbf{prod}_R(\textbf{\textit{f}})      
.          \]
A functional $f\in {\textbf{\textsf{L}}^\vee}^\ast[P]$ is called \emph{$R$-semisimple} if it is in the image of this embedding. If $R=P$, then $\textbf{\textit{f}}=f_1\otimes \dots \otimes f_k$ is a tensor product of simple functionals 
\[
f_j\in {\textbf{\textsf{L}}^\vee}^\ast[^{I\! } S_j]\cong {\textbf{\text{L}}^\vee}^\ast[S_j]
.\] 
In this case, we can make the identification ${\textbf{\textsf{L}}^\vee}^\ast[^{I\! } S_j]= {\textbf{\text{L}}^\vee}^\ast[S_j]$ to define
\[
\text{prod}_P:  {\textbf{\text{L}}^\vee}^\ast[S_1]\otimes \dots \otimes  {\textbf{\text{L}}^\vee}^\ast[S_k]\hookrightarrow  {\textbf{\textsf{L}}^\vee}^\ast[P], \qquad  \ \    \textbf{\textit{f}} \mapsto  \textbf{prod}_P(\textbf{\textit{f}}) 
. \]
A functional $f\in {\textbf{\textsf{L}}^\vee}^\ast[P]$ is called \emph{semisimple} if it is $P$-semisimple, i.e. if it is in the image of $\text{prod}_P$.

Let
\[   
 \textbf{\textsf{L}}^\vee[P|R]:=  \  \bigslant{ \textbf{\textsf{L}}^\vee[P] }{  \text{ker}\ \textbf{proj}_R   } 
.\]
Then $\textbf{proj}_R$ and $\textbf{prod}_R$ induce isomorphisms 
\[      
\textbf{\textsf{L}}^\vee[P|R] \cong   \textbf{\textsf{L}}^\vee[^{I\! }P_1]\otimes \dots \otimes \textbf{\textsf{L}}^\vee[^{I\! }P_k]  
\]
and
\[   {\textbf{\textsf{L}}^\vee}^\ast[P|R]  :=  \Hom(\textbf{\textsf{L}}^\vee[P|R], \Bbbk)      \cong   {\textbf{\textsf{L}}^\vee}^\ast[^{I\! }P_1]\otimes \dots \otimes  {\textbf{\textsf{L}}^\vee}^\ast[^{I\! }P_k]  . \]
In particular, ${\textbf{\textsf{L}}^\vee}^\ast[P|R]$ is naturally the space of $R$-semisimple functionals. \medskip

We now give an explicit description of $\text{ker}\ \textbf{proj}_R$. We continue to let $R$ and $P$ be partitions of $I$ with $P$ finer than $R$. We say that two adjoint faces are Steinmann $R$-adjacent if they are adjacent via a hyperplane which is not $R$-semisimple. We make the following equivalent definition.  

\begin{definition}
For distinct adjoint faces $\mathtt{X}_1,\mathtt{X}_2\in \textbf{\textsf{L}}^\vee[P]$, we say that $\mathtt{X}_1$ is \hbox{\emph{Steinmann $R$-adjacent}} to $\mathtt{X}_2$ if there exists a subset $E\subset I$ with $E\neq \emptyset \mod P$ which is not $R$-semisimple, and such that $s_{\mathtt{X}_1}$ and $s_{\mathtt{X}_2}$ differ only on the subsets which are equal to $E\mod  P$ or $I-E \mod  P$. 
\end{definition}

If $\mathtt{X}_1$ and $\mathtt{X}_2$ are $R$-semisimple via $E$, then it follows that 
\[   
s_{\mathtt{X}_2}(S) = 
\begin{cases}
-s_{\mathtt{X}_1}(S)  &\quad\text{if $S= E\mod P  \quad $  or  $\quad S= I-E \mod P   $} \\
\ \ s_{\mathtt{X}_1}(S)  &\quad\text{otherwise.} 
\end{cases}
\]
We say that $\mathtt{X}_1$ is \emph{Steinmann $R$-equivalent} to $\mathtt{X}_2$ if there exists a finite sequence of consecutively Steinmann $R$-adjacent adjoint faces which starts with $\mathtt{X}_1$ and terminates with $\mathtt{X}_2$. For Steinmann $P$-adjacency and Steinmann $P$-equivalence, we just say \emph{Steinmann adjacent} and \emph{Steinmann equivalent} respectively.   

In the following proof, see \cite[Section 2.1.7]{aguiar2017topics} for the notion of walls in hyperplane arrangements. 

\begin{thm}\label{lem}
Let $R$ and $P$ be partitions of $I$, with $P$ finer than $R$. Let $\mathtt{X}_1,\mathtt{X}_2\in \textbf{\textsf{L}}^\vee[P]$ be distinct adjoint faces. Then $\mathtt{X}_1$ is Steinmann $R$-equivalent to $\mathtt{X}_2$ if and only if 
\[\textbf{proj}_R  (\mathtt{X}_1)= \textbf{proj}_R  (\mathtt{X}_2).\]
In other words, Steinmann $R$-adjacency generates $\text{ker}\ \textbf{proj}_R$, and so $\textbf{\textsf{L}}^\vee[P|R]$ is naturally the quotient of $\textbf{\textsf{L}}^\vee[P]$ by Steinmann $R$-adjacency. 
\end{thm} 
\begin{proof}
Directly from the definition of projections of adjoint faces, we see that 
\[\textbf{proj}_R  (\mathtt{X}_1)= \textbf{proj}_R  (\mathtt{X}_2)\] 
if and only if $s_{\mathtt{X}_1}$ and $s_{\mathtt{X}_2}$ take the same value on subsets of $I$ which are $R$-semisimple. The signs of Steinmann $R$-equivalent adjoint faces must agree on subsets which are $R$-semisimple because, in the definition of Steinmann $R$-adjacency, the sign is altered only on subsets which are not $R$-semisimple. Thus, Steinmann $R$-equivalence implies equal projections. 

Conversely, suppose that the signs $s_{\mathtt{X}_1}$ and $s_{\mathtt{X}_2}$ take the same values on subsets of $I$ which are $R$-semisimple. Then, since $\mathtt{X}_1$ and $\mathtt{X}_2$ are distinct and yet are not separated by any $R$-semisimple hyperplanes, there must exist a wall of (the face corresponding to) $\mathtt{X}_1$ which is not $R$-semisimple and which separates $\mathtt{X}_1$ and $\mathtt{X}_2$. Let this separating wall be defined by $\lambda_E=0$ for some proper and nonempty subset $E\subset I$. Then, move to the new adjoint face obtained from $\mathtt{X}_1$ by switching the sign of $s_{\mathtt{X}_1}$ on sets $S$ of the form $S= E\mod P$ or $S= I-E\mod P$. This new adjoint face is Steinmann $R$-adjacent to $\mathtt{X}_1$. We repeat this process until the newly obtained adjoint face is $\mathtt{X}_2$, which must happen after a finite number of iterations. This produces a sequence of consecutively Steinmann $R$-adjacent adjoint faces from $\mathtt{X}_1$ to $\mathtt{X}_2$, and so $\mathtt{X}_1$ is Steinmann $R$-equivalent to $\mathtt{X}_2$.
\end{proof}



\begin{cor} \label{cor}
A functional on adjoint faces is $R$-semisimple if and only if it is constant on Steinmann $R$-equivalence classes of adjoint faces. In particular, a functional on adjoint faces is semisimple if and only if it is constant on Steinmann equivalence classes of adjoint faces
\end{cor}

In terms of finite differences of functionals across hyperplanes, which we study next, \autoref{cor} characterizes $R$-semisimple functionals as functionals whose value does not change across hyperplanes which are not $R$-semisimple. In particular, a semisimple functional is equivalently a functional whose value changes only across semisimple hyperplanes.

\section{The Forest Derivative} \label{sec:forestderiv}

\noindent \hyperlink{foo}{We} define the notion of tree and forest we shall be using, and define a composition of forests. This gives forests the structure of a one-way category, which we call the category of partitions. In this category, forests represent processes of refining partitions. We associate linear maps to forests which evaluate finite differences of functionals on adjoint faces across semisimple flats. We show that this association has antisymmetry and satisfies the Jacobi identity, as interpreted on forests.

\subsection{Trees and Forests}

For us, a \emph{tree} over a finite set $S$ is a planar\footnote{\ i.e. a choice of left and right child is made at every node} full binary tree whose leaves are labeled bijectively with the blocks of a partition of $S$. 
\begin{figure}[H]
	\centering
	\includegraphics[scale=0.7]{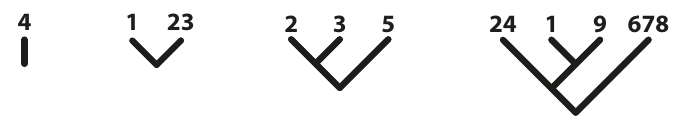}
	\caption{Some trees over sets of single digit integers. We adopt the convention that, for example, `$23$' denotes the set $\{2,3\}$. When drawing trees, we use the botanical convention rather than the uprooted convention, which is better suited to our needs.}
	\label{fig:tree}
\end{figure} 
\noindent A \emph{layered tree} $\mathcal{T}$ over a finite set $S_\mathcal{T}$ is a tree over $S_\mathcal{T}$, together with the structure of a linear ordering of the nodes of $\mathcal{T}$ such that if $v\in \mathcal{T}$ is a node on the geodesic from the root of $\mathcal{T}$ to another node $u \in \mathcal{T}$, then $v<u$. Layered trees are sometimes called trees with levels or increasing trees. We say that a layered tree is \emph{unlumped} if its leaves are labeled with singletons. 
\begin{figure}[H]
	\centering
	\includegraphics[scale=0.7]{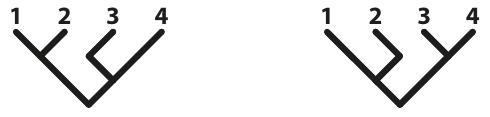}
	\caption{Schematic representations of two unlumped layered trees which have the same underlying delayered tree.}
	\label{fig:treelay}
\end{figure} 
\noindent We let $|\mathcal{T}|$ denote the number of leaves of $\mathcal{T}$. A \emph{stick} is a tree $\mathcal{T}$ with $|\mathcal{T}|=1$. We denote layered trees by nested products of sets $[\ \cdot\ ,\ \cdot\ ]$ when there is no ambiguity regarding the layering. For example, the trees in \autoref{fig:tree} have unique layerings, and may be denoted
\[   [4] \qquad  \ \ \  [1,23]\ \ \  \qquad  [[2,3],5]\ \ \  \qquad  [[ 24 ,[1,9] ] ,678 ].\]



Let $P=( S_1| \dots| S_k)$ be a partition of $I$. A \emph{layered forest} $\mathcal{F}= ( \mathcal{T}_1| \dots|\mathcal{T}_k )$ over $P$ is a set of trees such that $\mathcal{T}_j$ is a tree over $S_j$, together with the structure of a linear ordering of the nodes of the trees of $\mathcal{F}$ such that the restriction to each tree is a layered tree. 

For $Q$ a partition of $I$ which is finer than $P$, let $\textbf{\textsf{Lay}}[Q,P]$ denote the vector space of formal $\Bbbk$-linear combinations of layered forests over $P$, whose leaves are labeled bijectively with the blocks of $Q$. We have the space of unlumped layered trees over $I$, denoted
\[
\textbf{Lay}[I]:=  \textbf{\textsf{Lay}}\big [I,(I)\big ] 
. \]
We write $\mathcal{F}:P \leftarrow Q$ to indicate $\mathcal{F}\in \textbf{\textsf{Lay}}[Q,P]$. The arrow is intentionally in the direction of fusing the blocks of $Q$ together to form $P$.

Let $P=( S_1| \dots| S_k)$ be a partition of $I$, and let $\mathcal{T}$ be a tree over $S_m$ for some $1\leq m \leq k$. The \emph{$\mathcal{T}$-forest} over $P$ is the forest obtained by completing $\mathcal{T}$ with sticks labeled by the $S_j$, $j\neq m$. In contexts where there is no ambiguity, we denote this forest by $\mathcal{T}$. The $\mathcal{T}$-forests with $|\mathcal{T}|=2$ are called \emph{cuts}, and are denoted by $\mathcal{V}$. We denote by $\overline{\mathcal{V}}$ the cut which is obtained by switching the left and right branches of $\mathcal{V}$.

\begin{figure}[H]
	\centering
	\includegraphics[scale=0.7]{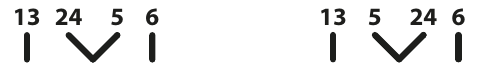}
	\caption{The cuts $\mathcal{V}=[24,5]$ and $\overline{\mathcal{V}}=[5,24]$ over the partition $(13|245|6)$}
	\label{fig:cut}
\end{figure} 

\begin{definition}
Let $\mathcal{F}_1:P \leftarrow Q$ and $\mathcal{F}_2:Q \leftarrow R$ be layered forests. The \emph{composition} 
\[
\mathcal{F}_1\circ \mathcal{F}_2:P \leftarrow R
\] 
of $\mathcal{F}_2$ with $\mathcal{F}_1$ is the layered forest which is obtained by identifying each leaf of $\mathcal{F}_1$ with the root node of the corresponding tree of $\mathcal{F}_2$, and requiring that $v_1$ is less than $v_2$ for all nodes $v_1\in \mathcal{F}_1$ and $v_2\in \mathcal{F}_2$. 
\end{definition}

Every layered forest $\mathcal{F}$ has a unique decomposition into cuts, corresponding to the linear ordering of the nodes of $\mathcal{F}$,
\[   
\mathcal{F}=  \mathcal{V}_{1} \circ  \dots \circ   \mathcal{V}_{l}         
 .   \]
The \emph{category of partitions} $\textsf{Lay}_I$ over $I$ is the linear one-way category with objects the partitions of $I$, hom-spaces formal linear combinations of layered forests
\[    
\Hom_{ \textsf{Lay}_I } (Q,P) :=  \textbf{\textsf{Lay}}[Q,P] 
,\] 
and morphism composition the linearization of layered forest composition,
\[      
\textbf{\textsf{Lay}}[R,Q] \otimes \textbf{\textsf{Lay}}[Q,P] \to \textbf{\textsf{Lay}}[R,P]
,\qquad 
\mathcal{F}_2\otimes \mathcal{F}_1 \mapsto \mathcal{F}_1\circ \mathcal{F}_2  
.\]

\begin{remark}
Categories of forests of this kind, subject to various relations, can often be interpreted as the structure for (symmetric May) operads with morphism composition providing the operadic composition, see \autoref{sec:lunes}.\footnote{\ note this correspondence between categories and operads, which appears in the theory of real hyperplane arrangements \cite[Section 15.9]{aguiar2017topics}, is slightly different to the correspondence which says that an operad is a symmetric multicategory with one object} In the case of the category of partitions, the structure preventing an operadic structure is the layering. Operads cannot `see' the layering because an operad models a forest as a tensor product of trees. However, a less structured version of operads was recently given in \cite[Chapter 4]{aguiar2020bimonoids} which formalizes the current situation. The linear category $\textsf{Lay}_I$ perspective we take in this paper is \cite[Section 4.2.3]{aguiar2020bimonoids}.
\end{remark}

The category of partitions is freely generated by cuts $\mathcal{V}$, which follows from the fact that every layered forest has a unique decomposition into cuts. See \autoref{sec:lunes} for an important interpretation of this category in terms of the braid arrangement. We also show in \autoref{sec:lunes} that the category of partitions acts on faces and top-lunes of both the braid arrangement and the adjoint braid arrangement.
 
Let $P=(S_1|\dots| S_k)$ be a partition of $I$, and let $C$ be a proper and nonempty subset of a block $S_m \in P$, for some $1\leq m \leq k$. Let $D=S_m- C$. Let $P'$ be the refinement of $P$ which is obtained by replacing the block $S_m$ with the blocks $C$ and $D$. Then the forests $\mathcal{F}$ of the form $\mathcal{F}:P \leftarrow P'$ are the cuts
\[\mathcal{V}=[C,D] \qquad \text{and} \qquad    \overline{\mathcal{V}}=[D,C].  \] 
We think of $ \mathcal{V}$ as refining $P$ by choosing $C$, and of $ \overline{\mathcal{V}}$ as refining $P$ by choosing $D$. 

In this way, a generic layered forest $\mathcal{F}:P \leftarrow Q$ describes a process of refining $P$ to give $Q$ by cutting blocks such that each time a block is cut, one of the two new blocks is favored. The favored block appears on the left branch of the forest, whereas the unfavored block appears on the right branch.

\begin{definition}
For $\mathcal{F}\in \textsf{Lay}_I$ a layered forest, the \emph{antisymmetrization} $\cA_I(\mathcal{F})\in \textsf{Lay}_I$ of $\mathcal{F}$ is the alternating sum of all the layered forests which are obtained by switching left and right branches at nodes of $\mathcal{F}$, with sign the parity of the number of switches. 
\end{definition} 

The antisymmetrization of layered forests defines a linear endofunctor on the category of partitions,
\[      \cA_I: \textsf{Lay}_I \to \textsf{Lay}_I, \qquad  \mathcal{F}\mapsto \cA_I(\mathcal{F}) .            \]




\subsection{The Forest Derivative} \label{sec:deriv}

\hyperlink{foo}{Given} an adjoint face $\mathtt{X}\in \textbf{\textsf{L}}^\vee[Q]$ and a cut \hbox{$\mathcal{V}=[C,D]:P \leftarrow Q$} of a block $S_m\in P$, let us denote by $\mathtt{X}^\mathcal{V}$ the adjoint face in $\textbf{\textsf{L}}^\vee[P]$ which has sign
\[   
s_{\mathtt{X}^{\mathcal{V}}}(S) := 
\begin{cases}
+ &\quad\text{if $S=C \mod P$} \\
- &\quad\text{if $S=D \mod P$}\\
s_\mathtt{X}(S)  &\quad\text{if $S\neq \emptyset \mod Q$} \\
0 &\quad\text{if $S= \emptyset \mod P$}.  \\
\end{cases}
\]
Notice that $\text{T}^\vee[Q]$ is the hyperplane of $\text{T}^\vee[P]$ which contains $\mathtt{X}$ as a top dimensional adjoint face, and $\mathtt{X}^{\mathcal{V}}$ and $\mathtt{X}^{\overline{\mathcal{V}}}$ are the two adjoint faces with support $\text{T}^\vee[P]$ which have $\mathtt{X}$ as a `facet'. 
\begin{figure}[H]
	\centering
	\includegraphics[scale=1.05]{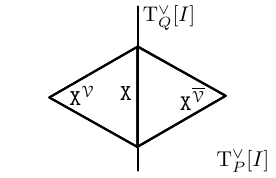}
	\label{fig:bracket}
\end{figure} 
\noindent In terms of Aguiar-Mahajan's general theory of real hyperplane arrangements \cite{aguiar2017topics}, $\mathtt{X}^\mathcal{V}$ is the composition (in the category of lunes) of the half-flat associated to $\mathcal{V}$, with $\mathtt{X}$. See \autoref{sec:lunes} for details.

\begin{remark} \label{arrow}
In the case $\textbf{\textsf{L}}^\vee[Q]=\textbf{\textsf{L}}^\vee[^{I\! } S]$ with $|S|=n-1$, and if we identify $\textbf{L}^\vee[S]=\textbf{L}^\vee[^{I\! } S]$, then we obtain two up-operators on the species $\textbf{L}^\vee$ called the Steinmann arrows, see e.g. \cite[Equation 20 and 21]{epstein2016}. 
\end{remark}

\begin{definition} \label{def:deriv}
Let $f \in {\textbf{\textsf{L}}^\vee}^\ast[P]$ be a functional on adjoint faces, and let $\mathcal{V}:P \leftarrow Q$ be a cut of $P$. The \emph{derivative} $\boldsymbol{\partial} _{\mathcal{V}} f$ of $f$ with respect to $\mathcal{V}$ is the functional in ${\textbf{\textsf{L}}^\vee}^\ast[Q]$ which is given by
\[       \boldsymbol{\partial} _{\mathcal{V}} f (\mathtt{X}):=  f(\mathtt{X}^\mathcal{V})- f(\mathtt{X}^{\overline{\mathcal{V}}}), \qquad \text{for} \quad      \mathtt{X}\in  \textbf{\textsf{L}}^\vee[Q] .  \] 
More generally, let $\mathcal{F}:P \leftarrow Q$ be any layered forest with decomposition into cuts
\[      \mathcal{F}=  \mathcal{V}_1 \circ  \dots \circ \mathcal{V}_l  .     \]
The \emph{forest derivative} $\boldsymbol{\partial} _{\mathcal{F}} f$ of $f$ with respect to $\mathcal{F}$ is the functional in ${\textbf{\textsf{L}}^\vee}^\ast[Q]$ which is given by the following composition of derivatives with respect to cuts,
\[        \boldsymbol{\partial} _\mathcal{F} f :=   \boldsymbol{\partial} _{\mathcal{V}_l}  (    \boldsymbol{\partial} _{\mathcal{V}_{l-1}} ( \dots   ( \boldsymbol{\partial} _{\mathcal{V}_2} (\boldsymbol{\partial} _{\mathcal{V}_1} f) )\dots ))   .    \]
\end{definition}

We shall sometimes call a derivative of the form $\boldsymbol{\partial} _{\mathcal{V}} f$ a \emph{first derivative}. See \autoref{sec:lunes} for a deeper perspective on the forest derivative, in terms of the category of Lie elements of the adjoint braid arrangement. 

In situations where a partition $P=( S_1| \dots| S_k)$ of $I$ is given, and $\mathcal{T}$ is a tree over $S_m$ for some $1\leq m\leq k$, we let $\boldsymbol{\partial} _\mathcal{T}$ denote the derivative with respect to the $\mathcal{T}$-forest over $P$ (recall that the $\mathcal{T}$-forest over $P$ is completion of $\mathcal{T}$ with sticks labeled by the blocks of $P$). 

The forest derivative defines a linear map of functionals on adjoint faces, given by
\[     \boldsymbol{\partial} _\mathcal{F} : {\textbf{\textsf{L}}^\vee}^\ast[P]\to  {\textbf{\textsf{L}}^\vee}^\ast[Q], \qquad   f\mapsto  \boldsymbol{\partial} _\mathcal{F} f.           \]
It is a direct consequence of the definition that the derivative respects forest composition; we have
\[               \boldsymbol{\partial} _{\mathcal{F}_1 \circ \mathcal{F}_2} = \boldsymbol{\partial} _{\mathcal{F}_2}  \circ \boldsymbol{\partial} _{ \mathcal{F}_1 } .       \]
The identities of $\textsf{Lay}_I$ are the forests of sticks. If $\mathcal{F}$ is a forest of sticks, then the decomposition of $\mathcal{F}$ into cuts is empty, and $\boldsymbol{\partial}_\mathcal{F}$ is the identity linear map. Therefore, the forest derivative defines a contravariant linear functor on the category of partitions into the category of vector spaces, given by
\[         \textsf{Lay}^\text{op}_I \to \textsf{Vec}, \qquad  \ \          P\mapsto {\textbf{\textsf{L}}^\vee}^\ast[P]     , \quad            \mathcal{F} \mapsto \boldsymbol{\partial}_\mathcal{F}.           \]


\subsection{The Dual Forest Derivative} 

We now describe the linear dual of the forest derivative. 

\begin{definition}\label{dualderiv}
Let $\mathtt{X}\in \textbf{\textsf{L}}^\vee[Q]$ be an adjoint face, and let $\mathcal{V}:P \leftarrow Q$ be a cut which produces $Q$. The \emph{dual derivative} $\boldsymbol{\partial}^\ast_\mathcal{V}\,  \mathtt{X}$ of $\mathtt{X}$ with respect to the cut $\mathcal{V}$ is the element of $\textbf{\textsf{L}}^\vee[P]$ which is given by
\[          \boldsymbol{\partial}^\ast_\mathcal{V}\,  \mathtt{X} := \mathtt{X}^\mathcal{V}-\mathtt{X}^{\overline{\mathcal{V}}}.              \]
More generally, let $\mathcal{F}:P \leftarrow Q$ be any layered forest with decomposition into cuts
\[      \mathcal{F}=  \mathcal{V}_1 \circ  \dots \circ \mathcal{V}_l  .    \]
The \emph{dual forest derivative} $\boldsymbol{\partial}^\ast_\mathcal{F}\,  \mathtt{X}$ of $\mathtt{X}\in \textbf{\textsf{L}}^\vee[Q]$ with respect to $\mathcal{F}$ is the element of $\textbf{\textsf{L}}^\vee[P]$ which is given by the following composition of dual derivatives with respect to cuts,
\[        \boldsymbol{\partial}^\ast_\mathcal{F}\,  \mathtt{X} :=     \boldsymbol{\partial} ^\ast_{\mathcal{V}_1}  (     \boldsymbol{\partial} ^\ast_{\mathcal{V}_{2}}   \dots    ( \boldsymbol{\partial} ^\ast_{\mathcal{V}_{l-1}}    (\boldsymbol{\partial} ^\ast_{\mathcal{V}_l} (\mathtt{X}))\dots ))   .     \]
\end{definition}

After extending linearly, the dual forest derivative is a linear map of linear combinations of adjoint faces, given by
\[     \boldsymbol{\partial} ^\ast_\mathcal{F} : \textbf{\textsf{L}}^\vee[Q]  \to \textbf{\textsf{L}}^\vee[P], \qquad   \mathtt{X}\mapsto  \boldsymbol{\partial}^\ast_\mathcal{F}\,  \mathtt{X}.           \]
It is a direct consequence of the definition that the dual derivative respects forest composition,
\[               \boldsymbol{\partial} ^\ast_{\mathcal{F}_1 \circ \mathcal{F}_2} = \boldsymbol{\partial} ^\ast_{\mathcal{F}_1}  \circ \boldsymbol{\partial} ^\ast_{ \mathcal{F}_2 } .       \]
Notice that $\boldsymbol{\partial} ^\ast _\mathcal{F}$ is just the linear dual of $\boldsymbol{\partial} _\mathcal{F}$, we have
\[            \boldsymbol{\partial} _\mathcal{F} f (\mathtt{X})  = f ( \boldsymbol{\partial}_\mathcal{F}^\ast\,  \mathtt{X}  )    .      \qedhere   \]
The dual forest derivative defines a covariant linear functor on the category of partitions into the category of vector spaces, given by
\[         \textsf{Lay}_I \to \textsf{Vec}, \qquad   \ \            P\mapsto \textbf{\textsf{L}}^\vee[P]     , \quad       \mathcal{F} \mapsto \boldsymbol{\partial}^\ast_\mathcal{F}.           \]

For $\mathcal{F}=  \mathcal{V}_1 \circ  \dots \circ \mathcal{V}_l$, let
\begin{equation*}
\mathtt{X}^\mathcal{F}:=   { \big ({{ (\mathtt{X}^{\mathcal{V}_l})  }}^{\dots}\big )}   ^{\mathcal{V}_1} .   
\end{equation*}
Let us extend the map $\mathcal{F} \mapsto \mathtt{X}^\mathcal{F}$ linearly to formal linear combinations of forests. Then directly from the definition of $ \boldsymbol{\partial} ^\ast_\mathcal{F}$, we see that  \begin{figure}[t]
	\centering
	\includegraphics[scale=0.6]{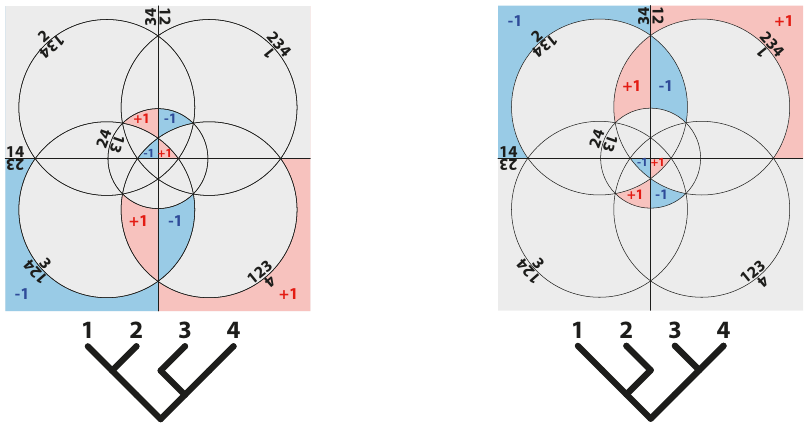}
	\caption{The dual forest derivative $\boldsymbol{\partial}^\ast_\mathcal{F} \mathtt{X}$ for $\mathtt{X}$ equal to the zero-dimensional adjoint face, and $\mathcal{F}$ equal to the two layerings of $[[1,2],[3,4]]$, depicted on the stereographic projection of the Steinmann planet. For unlumped forests, $\boldsymbol{\partial}^\ast_\mathcal{F} \mathtt{X}$ is the Lie element of $\mathcal{F}$ (see \autoref{sec:lunes}).}
	\label{fig:layer}
\end{figure} 
\[  \boldsymbol{\partial}^\ast_\mathcal{F}\,  \mathtt{X}  =  \mathtt{X}^{\cA_I( \mathcal{F})}.          \]
In particular, we have
\[            \boldsymbol{\partial} _\mathcal{F} f (\mathtt{X})  =  f( \mathtt{X}^{   \cA_I( \mathcal{F})   })    .         \] 
Note that the derivative does depend upon the layering of forests, see \autoref{fig:layer}.  


\subsection{Lie Properties}  \label{sec:lieprop}

We now show that the forest derivative satisfies the Lie axioms of antisymmetry and the Jacobi identity, as interpreted on layered forests. We first give two examples, the first showing antisymmetry holding for $n=2$, and the second showing the Jacobi identity holding for $n=3$. 

\begin{ex}
Let $I=\{1,2\}$. Let $\mathtt{X}$ be the central zero-dimensional adjoint face. Then
\[      \boldsymbol{\partial} ^\ast_{[1,2]} \mathtt{X} +  \boldsymbol{\partial} ^\ast_{[2,1]} \mathtt{X} =0      .              \]
Schematically, we have 
\begin{figure}[H]
	\centering
	\includegraphics[scale=0.6]{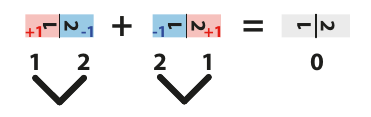}
\end{figure} 
\end{ex}

\begin{ex} \label{jacid}
Let $I=\{1,2,3\}$. Again, let $\mathtt{X}$ be the central zero-dimensional adjoint face. Then
\[      \boldsymbol{\partial} ^\ast_{[[1,2],3   ]} \mathtt{X} +  \boldsymbol{\partial} ^\ast_{[[3,1], 2 ]} \mathtt{X} +\boldsymbol{\partial} ^\ast_{[[2,3], 1 ]} \mathtt{X}=0      .              \]
Schematically, we have
\begin{figure}[H]
	\centering
	\includegraphics[scale=0.6]{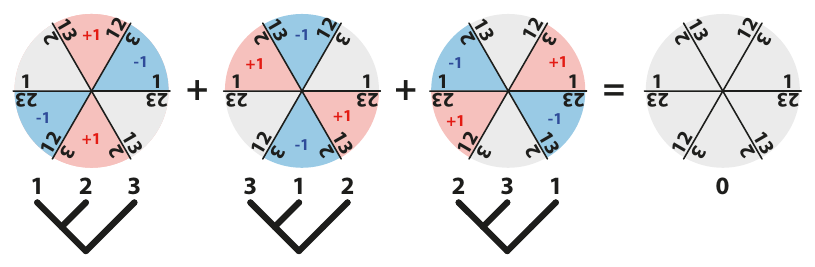}
\end{figure} 
\end{ex} 

\begin{thm}\label{thm:lie}
Let $P=( S_1|\dots| S_k)$ be a partition of $I$, and let $S_m$, $1\leq m\leq k$, be a block of $P$. 

\begin{enumerate}
\item
For trees $\mathcal{T}_1$ and $\mathcal{T}_2$ such that $(S_{\mathcal{T}_1}| S_{\mathcal{T}_2})$ is a partition of $S_m$, we have
\[  \boldsymbol{\partial} _{[\mathcal{T}_1,\mathcal{T}_2]}+\boldsymbol{\partial} _{[\mathcal{T}_2,\mathcal{T}_1]}=0.  \]
\item
For trees $\mathcal{T}_1$, $\mathcal{T}_2$ and $\mathcal{T}_3$ such that $( S_{\mathcal{T}_1}| S_{\mathcal{T}_2}| S_{\mathcal{T}_3} )$ is a partition of $S_m$, we have
\[ \boldsymbol{\partial} _{[[\mathcal{T}_1,\mathcal{T}_2],\mathcal{T}_3]}   +\boldsymbol{\partial} _{[[\mathcal{T}_3,\mathcal{T}_1],\mathcal{T}_2]}  +\boldsymbol{\partial} _{[[\mathcal{T}_2,\mathcal{T}_3],\mathcal{T}_1]}  =0 .\] 
\end{enumerate}
\end{thm}
\begin{proof}
We first prove antisymmetry (1). Since the derivative respects forest composition, we have
\[         \boldsymbol{\partial} _{[\mathcal{T}_1,\mathcal{T}_2]}+\boldsymbol{\partial} _{[\mathcal{T}_2,\mathcal{T}_1]}    =  ( \boldsymbol{\partial} _{  [ S_{\mathcal{T}_1} ,  S_{\mathcal{T}_2}  ] }+  \boldsymbol{\partial} _{  [ S_{\mathcal{T}_2} ,  S_{\mathcal{T}_1}  ] } )\circ \boldsymbol{\partial} _{   (\mathcal{T}_1|\mathcal{T}_2  )  } .\footnote{\ recall that $(\mathcal{T}_1|\mathcal{T}_2  )$ denotes the forest containing the trees $\mathcal{T}_1$ and $\mathcal{T}_2$}   \]
Therefore it is enough to check the case where $[\mathcal{T}_1,\mathcal{T}_2]$ is a cut $\mathcal{V}$. Then, rewriting in terms of the dual derivative, we have 
\[   (\boldsymbol{\partial} _{\mathcal{V}}+\boldsymbol{\partial} _{\overline{\mathcal{V}}})f(\mathtt{X})=   f(  \boldsymbol{\partial}^\ast _{\mathcal{V}}\,  \mathtt{X}  + \boldsymbol{\partial}^\ast _{\overline{\mathcal{V}}}\,  \mathtt{X}   ) = f( \mathtt{X}^{\mathcal{V}}-\mathtt{X}^{\overline{\mathcal{V}}}+\mathtt{X}^{\overline{\mathcal{V}}}-\mathtt{X}^{\mathcal{V}} )=f(0)=0      .  \]
We now prove the Jacobi identity (2). Since the derivative respects forest composition, we have
\begin{align*}
&\boldsymbol{\partial} _{[[\mathcal{T}_1,\mathcal{T}_2],\mathcal{T}_3]}   +\boldsymbol{\partial} _{[[\mathcal{T}_3,\mathcal{T}_1],\mathcal{T}_2]}  +\boldsymbol{\partial} _{[[\mathcal{T}_2,\mathcal{T}_3],\mathcal{T}_1]} \\[6pt]
=&(\boldsymbol{\partial} _{[[ S_{\mathcal{T}_1}, S_{\mathcal{T}_2}], S_{\mathcal{T}_3}]}   +\boldsymbol{\partial} _{[[ S_{\mathcal{T}_3}, S_{\mathcal{T}_1}], S_{\mathcal{T}_2}]}  +\boldsymbol{\partial} _{[[ S_{\mathcal{T}_2}, S_{\mathcal{T}_3}], S_{\mathcal{T}_1}]})\circ \boldsymbol{\partial} _{   (\mathcal{T}_1|\mathcal{T}_2|\mathcal{T}_3 ) }. 
\end{align*}
Therefore it is enough to check the case where $[[\mathcal{T}_1,\mathcal{T}_2],\mathcal{T}_3]=  [[ S,T  ],U ] $, for $(S|T|U)$ a partition of $S_m$. Then, rewriting in terms of the dual derivative, we have  
\begin{align*}
 &(\boldsymbol{\partial} _{ [[S,T],U]  } +\boldsymbol{\partial} _{ [[U,S],T]  } +\boldsymbol{\partial} _{ [[T,U],S]  })f(\mathtt{X}) \\[6pt ]
 =   f&(    \boldsymbol{\partial}^\ast_{[[S,T],U]} \mathtt{X}  +   \boldsymbol{\partial}^\ast_{[[U,S],T]} \mathtt{X}+ \boldsymbol{\partial} ^\ast_{[[T,U],S]} \mathtt{X}     ) \\[6pt]
 =f& (         \mathtt{X}^{ [[ S,T ],U ]  }   - \mathtt{X}^{ [[ T,S ],U ]  } -\mathtt{X}^{ [U,[ S,T ] ]  } +\mathtt{X}^{ [U,[ T,S ] ]  }        \\[6pt]
  +&        \mathtt{X}^{ [[ U,S ],T ]  } - \mathtt{X}^{ [[ S,U ],T ]  } -\mathtt{X}^{ [T,[ U,S ] ]  } +\mathtt{X}^{ [T,[ S,U ] ]  }  \\[6pt]
   +&          \mathtt{X}^{ [[ T,U ],S ]  }   - \mathtt{X}^{ [[ U,T ],S ]  } -\mathtt{X}^{ [S,[ T,U ] ]  } +\mathtt{X}^{ [S,[ U,T ] ]  }    )     .   
\end{align*}
Directly from the definition of the dual derivative, we see that
\[   \mathtt{X}^{ [[S,T],U] }=\mathtt{X}^{ [[S,U],T] }, \quad \mathtt{X}^{ [T,[S,U]] }=\mathtt{X}^{ [S,[T,U]] } , \quad  \mathtt{X}^{ [[T,S],U] }=\mathtt{X}^{ [[T,U],S] },  \]
\[   \mathtt{X}^{ [U,[T,S]] }=\mathtt{X}^{ [T,[U,S]] }, \quad \mathtt{X}^{ [[U,S],T] }=\mathtt{X}^{ [[U,T],S] } , \quad  \mathtt{X}^{ [S,[U,T]] }=\mathtt{X}^{ [U,[S,T]] }.  \]
Therefore
\[   f(    \boldsymbol{\partial} ^\ast_{[[S,T],U]} \mathtt{X}  +   \boldsymbol{\partial} ^\ast_{[[U,S],T]} \mathtt{X}+ \boldsymbol{\partial} ^\ast_{[[T,U],S]} \mathtt{X}     )=0. \qedhere  \]
\end{proof}


\section{The Action of Lie Elements on Faces} \label{sec:lunes}

\noindent \hyperlink{foo}{We} describe the forest derivative of \autoref{def:deriv} in the context of the general theory of real hyperplane arrangements, as developed in \cite{aguiar2017topics}. We show that the dual forest derivative is obtained by representing certain Lie elements of the adjoint braid arrangement with layered binary trees, and then composing with the contravariant internal hom-functor which is represented by the zero-dimensional flat. This hom-functor constructs the action of Lie elements on faces. We also show how the dual forest derivative is the geometric analog of the regular right action of the Lie operad on the associative operad. 

\subsection{Categories of Lunes}

Let $\text{T}[I]$ denote the quotient of $\bR^I$ by the relation $\lambda_I=0$, 
\[         \text{T}[I]:=   \bR^I/\bR\lambda_I  .     \]
Thus, to obtain $\text{T}[I]$ from $\bR^I$, we consider functions $\lambda:I\to \bR$ only up to translations of $\bR$. Notice that $\text{T}[I]$ is naturally the $\bR$-linear dual space of $\text{T}^\vee[I]$. Indeed, we have a perfect pairing, given by
\[       \text{T}^\vee[I] \times  \text{T}[I]\to \bR, \qquad (h,\lambda)\mapsto \la  h,\lambda  \ra:= \sum_{i\in I}h_i \,  \lambda(i)    .  \]  
This is well-defined; since $\sum_{i\in I} h_i=0$, we have that translating $\lambda$ does not change the value of the pairing $\la  h,\lambda  \ra$.

We associate to each partition $P=( S_1| \dots| S_k)$ of $I$ the linear subspace $\text{T}[P]$ of $\text{T}[I]$ which is given by
\[         \text{T}[P]:= (\text{T}^\vee[P])^\perp= \big \{  \lambda\in  \text{T}[I]: \la h, \lambda \ra=0\ \text{for all}\ h\in \text{T}^\vee[P]  \big   \}.       \] 
Equivalently, $\text{T}[P]$ consists of those functions $\lambda:I\to \bR$ which are constant on the blocks of $P$. A \emph{reflection hyperplane} is a hyperplane of $\text{T}[I]$ of the form $\text{T}[P]$, which happens if $P$ is a partition with exactly one non-singleton block of cardinality two. The hyperplane arrangement consisting of the reflection hyperplanes in $\text{T}[I]$ is called the \emph{braid arrangement} over $I$, which we denote by $\mathrm{Br}[I]$. The chambers of the braid arrangement are called (type $A$) \emph{Weyl chambers}. The subspaces $\text{T}[P]$, as $P$ ranges over partitions of $I$, are the flats of $\mathrm{Br}[I]$. \medskip

We associate to a cut of the form
\[\mathcal{V}=[C,D]:P \leftarrow Q\]
the closed half-space of $\text{T}[Q]$ which is given by
\[      \big  \{ \lambda \in \text{T}[Q] :  \la h_{i_1 i_2}, \lambda\ra \geq 0 \text{ for any } i_1\in C \text{ and } i_2\in D   \big \}   .                \]
Notice that this associates to $\overline{\mathcal{V}}$ the complementary closed halfspace. Under this association, we may view the category of partitions $\textsf{Lay}_I$ as the category which is freely generated by half-flats of the braid arrangement $\mathrm{Br}[I]$. 
We also associate to $\mathcal{V}$ the closed half-space of $\text{T}^\vee[P]$ which is given by
\[      \big  \{ h\in \text{T}^\vee[P] : \la h, \lambda_{C}\ra\geq 0   \big \}   .                \]
Notice that this associates to $\overline{\mathcal{V}}$ the complementary closed halfspace. \medskip



Associated to any real hyperplane arrangement is its category of lunes \cite[Section 4.4]{aguiar2017topics}, which has objects the flats, and morphisms the so-called lunes. A lune is any intersection of closed half-spaces which is also a chamber for some restriction of the arrangement to an interval of flats. The convention is that the source of a lune is its supporting flat, or `case', and the target of a lune is its base flat. This means that morphisms in the category of lunes travel in the direction of reducing dimension, and increasing codimension. Thus, braid arrangement lunes point in the direction of fusing blocks together (algebra), whereas adjoint braid arrangement lunes point in the direction of cutting blocks (coalgebra). The composition of lunes $L \circ M$ is the unique appropriate lune $N$ such that $L \subseteq  N \subseteq   M$ \cite[Proposition 4.30]{aguiar2017topics}. 

Let $\textsf{Ass}_I$ denote the linearized category of lunes of the braid arrangement $\mathrm{Br}[I]$, and let $\textsf{Ass}^\vee_I$ denote the linearized category of lunes of the adjoint braid arrangement $\mathrm{Br}^\vee[I]$. The structure of $\textsf{Ass}_I$ is essentially that of the associative operad \cite[Section 6.5.10]{aguiar2017topics}. The object $\textsf{Ass}^\vee_I$ seems to be less structured, however for us it is merely the substrate. 

The two associations of half-spaces to cuts, just defined, are functions on the free generators of $\textsf{Lay}_I$ into lunes of slack-$1$, and so define two linear functors
\[   \pi_I: \textsf{Lay}_I \twoheadrightarrow   \textsf{Ass}_I \qquad     \text{and} \qquad           \pi_I^\vee:  \textsf{Lay}^\text{op}_I  \to   \textsf{Ass}^\vee_I.\]
The functor $\pi_I$ is the quotient of $\textsf{Lay}_I$ by delayering and debracketing, i.e. $\pi_I$ sends a layered forest $\mathcal{F}$ to the lune corresponding to the composite set composition of $I$ which forms the canopy of $\mathcal{F}$. In particular, $\pi_I$ is surjective. 

We call a half-flat of the adjoint braid arrangement \emph{semisimple} if both its case and base are semisimple flats of $\text{T}^\vee[I]$. The image of $\pi_I^\vee$ is the subcategory of $\textsf{Ass}^\vee_I$ which is generated by semisimple half-flats.


\subsection{Categories of Lie Elements}

\begin{figure}[t]
	\centering
	\includegraphics[scale=0.6]{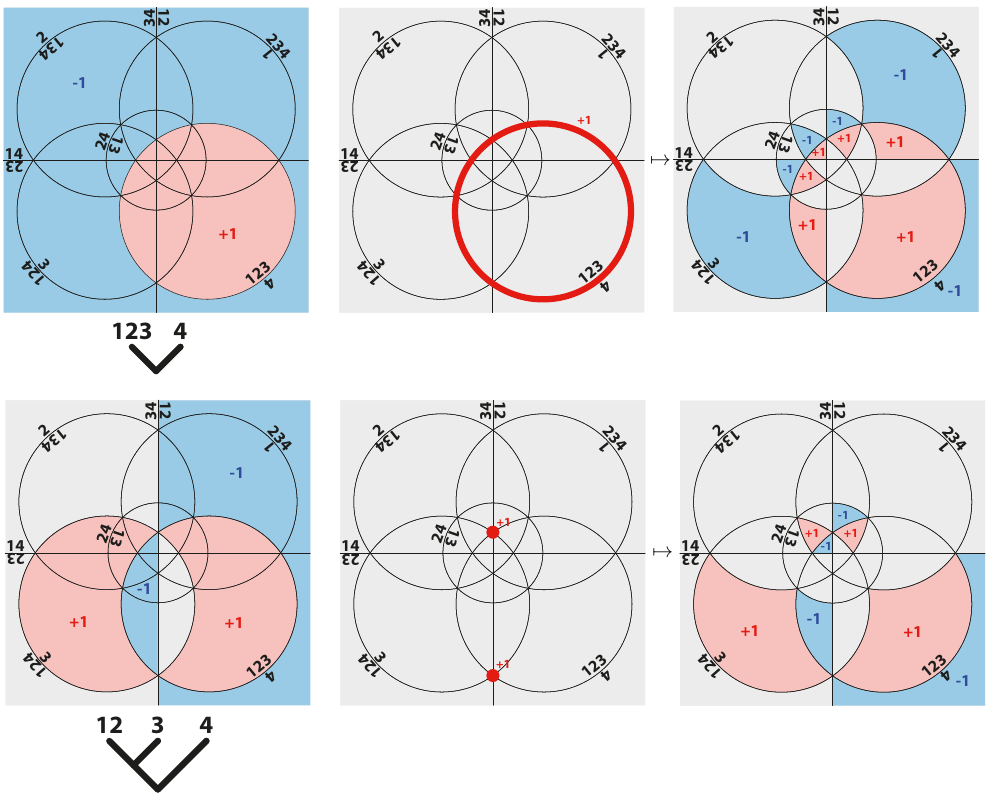}
	\caption{The action of the Lie elements of $[123,4]$ and $[[12,3],4]$ on flats; that is, if a Lie element has source the partition $P$, then we take the action of the Lie element on all the adjoint faces with support $\text{T}^\vee[P]$. The resulting linear combination of adjoint faces was called the antiderivative of a tree by Ocneanu in \cite{oc17}, where it was used to express the tree derivative $\boldsymbol{\partial}_\mathcal{T}$ as an inner product in the unlumped case $\mathcal{T}\in \textbf{Lay}[I]$.}  
	\label{fig:lieelements}
\end{figure}

Associated to any real hyperplane arrangement is its category of Lie elements \cite[Section 10.6.2]{aguiar2017topics}, which is a linear subcategory of the category of lunes. The embedding of Lie elements in lunes generalizes the embedding of the (positive) Lie operad in the associative operad 
\[\textbf{Lie} \hookrightarrow \textbf{Ass},\] 
which is recovered by specializing to the braid arrangement. Note that the category of Lie elements has morphisms not only the Lie elements of the arrangement, but also the Lie elements of all possible restrictions of the arrangement.\footnote{\ analogously, the category of lunes has morphisms not only the chambers, but also the chambers of all possible restrictions of the arrangement}

Let us denote by $\textsf{Lie}_I$ and $\textsf{Lie}^\vee_I$ the categories which are the images of the functors $\pi_I\circ \cA_I$ and $\pi^\vee_I\circ \cA_I$ respectively. Then $\textsf{Lie}_I$ is the category of Lie elements of $\mathrm{Br}[I]$, and $\textsf{Lie}^\vee_I$ is the subcategory of the category of Lie elements of $\mathrm{Br}^\vee[I]$ which is generated by differences of complimentary semisimple half-flats.

Taking certain hom-functors on $\textsf{Ass}_I$ or $\textsf{Ass}^\vee_I$ define two important actions of their lunes, one on chambers under flats (faces), and one on chambers over flats (top-lunes). Let $I$ also denote the partition of $I$ into singletons $(i_1|\dots|i_n)$, and let $(I)$ denote the partition with one block as usual. Then, over the braid arrangement, $\Hom_{  \textsf{Ass}_I }(I,-)$ is the action on top-lunes and $\Hom_{  \textsf{Ass}_I }(-,(I))$ is the action on faces, whereas over the adjoint braid arrangement, $\Hom_{  \textsf{Ass}^\vee_I }((I),-)$ is the action on top-lunes and $\Hom_{  \textsf{Ass}^\vee_I }(-,I)$ is the action on faces.
\bgroup
\def\arraystretch{1.35}  
\begin{table}[H] 
\begin{tabular}{|c|c|c|}
\hline
  &   \begin{tabular}{@{}c@{}}action on top-lunes\end{tabular} & \begin{tabular}{@{}c@{}}action on faces\end{tabular}   \\ \hline
 \begin{tabular}{@{}c@{}}braid arrangement \end{tabular}  &   $\Hom_{  \textsf{Ass}_I }(I,-)$    &  $\Hom_{  \textsf{Ass}_I }(-,(I))$  \\ \hline
 \begin{tabular}{@{}c@{}}adjoint braid arrangement \end{tabular}  &  $\Hom_{  \textsf{Ass}^\vee_I }((I),-)$   &   $\Hom_{  \textsf{Ass}^\vee_I }(-,I)$    \\ \hline
\end{tabular}
\end{table}
\egroup
\noindent We fix our geometric perspective, to consider only the actions on faces. We collect all our functors in the following diagram.

\begin{center}
\begin{tikzcd}[column sep=large,row sep=large] 
  \textsf{Vec}  & \textsf{Ass}_I   \arrow[l,   "\Hom(-\comma (I) )"']  & \textsf{Lay}_I  	\arrow[d, "\cA_I"]   	\arrow[l, rightarrow, "\pi_I"'] 	\arrow[r, rightarrow,  "\pi^\vee_I"]& 	\textsf{Ass}^\vee_I 		\arrow[r,  "\Hom(-\comma I)"] & 	  \textsf{Vec} \\
&     \textsf{Lie}_I    	\arrow[u, hook]     & \textsf{Lay}_I   	\arrow[l, twoheadrightarrow,  "\pi_I"'] 	\arrow[r, twoheadrightarrow,  "\pi^\vee_I"]&  	\textsf{Lie}^\vee_I 	\arrow[u, hook] 
\end{tikzcd}
\end{center}
For any adjoint face $\mathtt{X} \in \textsf{Ass}^\vee_I$ with semisimple support, and any cut $\mathcal{V}$ which produces the partition of the support of $\mathtt{X}$, we have
\[           \mathtt{X}^\mathcal{V}=  \mathtt{X}  \circ  \pi_I^\vee(\mathcal{V})   .\footnote{\ composition is in the category of lunes $\textsf{Ass}^\vee_I$} \] 
Then, it follows directly from the definitions that the actions for the adjoint braid arrangement have the following descriptions,
\begin{align*}
&\textsf{Lay}_I      \xrightarrow{\ \ \pi_I^\vee\ \ }    \textsf{Ass}^\vee_I  \ \xrightarrow{  \Hom(-,I)   }\textsf{Vec},&      \mathcal{F}&\mapsto (\mathtt{X}\mapsto \mathtt{X}^\mathcal{F})  \\ \intertext{and} 
&\textsf{Lay}_I    \xtwoheadrightarrow{\pi_I^\vee\circ  \cA_I}     \textsf{Lie}^\vee_I \   \xrightarrow{  \Hom(-,I)   }\textsf{Vec},&   \mathcal{F}&\mapsto \boldsymbol{\partial}^\ast_\mathcal{F}. 
\end{align*}
In particular, the forest derivative is obtained by composing the functor $\pi_I^\vee\circ  \cA_I$ with linear duality $\textsf{Vec}\to \textsf{Vec}^{\text{op}}$. 

On the braid arrangement side, $\pi_I\circ \cA_I$ is the quotient by antisymmetry, the Jacobi identity, and delayering. The action obtained by composing $\pi_I$, respectively $\pi_I\circ \cA_I$, with $\Hom(-\comma (I) )$ is the structure of the regular \emph{right} action of the associative operad $\Ass$, respectively the Lie operad $\Lie$, on $\Ass$. In particular, by composing $\pi_I\circ \cA_I$ with linear duality $\textsf{Vec}\to \textsf{Vec}^{\text{op}}$, we obtain the structure of the dual right coaction
\[ \textbf{Ass}^\ast \to   \textbf{Ass}^\ast \circ \Lie^\ast  . \] 
This is the geometric analog of the forest derivative on the braid arrangement, in the sense that it is the discrete differentiation of functions on Weyl chambers across reflection hyperplanes. For the regular left actions of $\Ass$ and $\Lie$ on $\Ass$, one should compose $\pi_I$ and $\pi_I\circ \cA_I$ with $\Hom(I, - )$.

On the adjoint braid arrangement side, \autoref{thm:lie} says that $\pi_I^\vee \circ \cA_I$ factors through the Lie axioms of antisymmetry and the Jacobi identity, but not necessarily delayering. We obtain delayering next, by restricting to functionals whose derivatives are semisimple. 

\section{Semisimple Differentiability and the Steinmann Relations} \label{semisimple}

\noindent \hyperlink{foo}{We} study the relationship between the forest derivative and the property of semisimplicity for functionals on adjoint faces. The derivative does not preserve semisimplicity, however we do show that the derivative of a product of functionals factorizes as a product of derivatives. Crucially, by restricting to functionals whose forest derivatives are semisimple, we are able to delayer forests, and thus obtain algebraic structure in species. The Steinmann relations of axiomatic quantum field theory are easily seen to be equivalent to the property that the first derivatives of functionals are semisimple. We show that the Steinmann relations are equivalent to the property that all derivatives are semisimple, i.e. to conclude that all of a functionals derivatives are semisimple, it is enough to check the first derivatives. 

\subsection{Derivatives of Products of Functions} 

\hyperlink{foo}{Let} $\mathcal{F}=( \mathcal{T}_1| \dots| \mathcal{T}_k):P \leftarrow Q$ be a layered forest, and let $P=(S_1|\dots|S_k)$. For $1\leq j \leq k$, let $Q_j$ denote the partition of $S_j$ which is the restriction of $Q$ to $S_j$, and let $^{I \! } Q_j$ denote the partition of $I$ which is the completion of $Q_j$ with singletons. Let us denote by 
\[\boldsymbol{\partial} _{j }:{\textbf{\textsf{L}}^\vee}^\ast[^{I\! } S_j] \to   {\textbf{\textsf{L}}^\vee}^\ast[^{I \! } Q_j] \] 
the derivative with respect to the completion of $\mathcal{T}_j$ with singleton sticks. Under the identifications 
\[
{\textbf{\textsf{L}}^\vee}^\ast[^{I\! } S_j]= {\textbf{\text{L}}^\vee}^\ast[S_j] 
\qquad  \text{and} \qquad 
{\textbf{\textsf{L}}^\vee}^\ast[^{I \! } Q_j]= {\textbf{\textsf{L}}^\vee}^\ast[Q_j],
\] 
the map $\boldsymbol{\partial} _{j }$ is the ordinary derivative 
\[  
\boldsymbol{\partial} _{\mathcal{T}_j }:{\textbf{\text{L}}^\vee}^\ast[S_j] \to       {\textbf{\textsf{L}}^\vee}^\ast[Q_j]
.\] 
Let us denote by $\otimes \boldsymbol{\partial} _ \mathcal{F}$ the tensor product of maps $\bigotimes_j \boldsymbol{\partial} _{_j}$, thus
\[\otimes\boldsymbol{\partial} _\mathcal{F}:  \bigotimes_j {\textbf{\textsf{L}}^\vee}^\ast[^{I\! } S_j]   \to\bigotimes_j{\textbf{\textsf{L}}^\vee}^\ast[^{I \! } Q_j], \qquad   f_1\otimes \dots \otimes f_k \mapsto \boldsymbol{\partial} _{1}   f_1\otimes \dots \otimes \boldsymbol{\partial} _{k}  f_k .	          \]
We now show that one can differentiate a semisimple functional by first differentiating the simple factors, and then taking their product. 

\begin{thm} \label{prop:delayer}
Let $\mathcal{F}=( \mathcal{T}_1| \dots| \mathcal{T}_k):P \leftarrow Q$ be a layered forest, and let $P=(S_1|\dots|S_k)$ be a partition of $I$. Then the following diagram commutes,
\[
\begin{tikzcd}
\bigotimes_j {\textbf{\textsf{L}}^\vee}^\ast[^{I \! } Q_j] \arrow[r, "\textbf{prod}_P"]   & {\textbf{\textsf{L}}^\vee}^\ast[Q]    \\	
{\bigotimes_j {\textbf{\textsf{L}}^\vee}^\ast[^{I\! } S_j]} \arrow[u, "\otimes \boldsymbol{\partial} _{\mathcal{F}}"] \arrow[r, "\textbf{prod}_P"'] &    {\textbf{\textsf{L}}^\vee}^\ast[P].  \arrow[u, "\boldsymbol{\partial} _\mathcal{F}"'] 
\end{tikzcd}
\]
\end{thm}

\begin{proof}
It follows directly from the definition of $\otimes \boldsymbol{\partial} _\mathcal{F}$ that
\[         \otimes \boldsymbol{\partial} _ {\mathcal{F}_1 \circ  \mathcal{F}_2} = \otimes \boldsymbol{\partial} _{\mathcal{F}_1}  \circ \otimes \boldsymbol{\partial} _{\mathcal{F}_2}                  . \]	
Therefore, since every forest is a composition of cuts, it is enough to consider the case where $\mathcal{F}$ is a cut $\mathcal{V}=[C,D]:P \leftarrow Q$ of a block $S_m$ of $P$. Let
\[\textbf{\textit{f}}=  f_1\otimes \dots \otimes f_k  \in  \bigotimes_j {\textbf{\textsf{L}}^\vee}^\ast[^{I\! } S_j]. \]
Put $f=\textbf{prod}_P(\textbf{\textit{f}})$. Then, for each adjoint face $\mathtt{X}\in \textbf{\textsf{L}}^\vee[Q]$, we have
\[    \boldsymbol{\partial} _\mathcal{V}  \big(  \textbf{prod}_P(\textbf{\textit{f}})\big)  (\mathtt{X})=    \boldsymbol{\partial} _\mathcal{V}  f(\mathtt{X})= f(\mathtt{X}^\mathcal{V})- f (\mathtt{X}^{\overline{\mathcal{V}}}) =  \prod_j f_j ( \mathtt{X}^{\mathcal{V}}|_{S_j})-\prod_j f_j (\mathtt{X}^{\overline{\mathcal{V}}}|_{S_j})      .     \]
For $j\neq m$, we have
\[\mathtt{X}^{\mathcal{V}}|_{S_j}=\mathtt{X}|_{S_j}=\mathtt{X}^{\overline{\mathcal{V}}}|_{S_j}.\] 
Therefore we can factor out the terms $j\neq m$, to obtain our required result,
\begin{align*}
    \boldsymbol{\partial} _\mathcal{V}  \big(  \textbf{prod}_P(\textbf{\textit{f}})\big)  (\mathtt{X})=    \boldsymbol{\partial} _\mathcal{V}  f(\mathtt{X}) &=\prod_{j\neq m} f_j(\mathtt{X}|_{S_j})\cdot \big ( f_m(\mathtt{X}^\mathcal{V}|_{S_m})-f_m (\mathtt{X}^{\overline{\mathcal{V}}}|_{S_m})\big ) \\
&=\textbf{prod}_P\Big ( \bigotimes_{j\neq m}  f_j \otimes \boldsymbol{\partial} _{m} f_m \Big ) (\mathtt{X}) \\
&= \textbf{prod}_P\big ( \otimes \boldsymbol{\partial}_\mathcal{V}  \textbf{\textit{f}}\ \big ) (\mathtt{X}) .\qedhere
\end{align*}
\end{proof}


\begin{definition}
A functional $f\in {\textbf{\textsf{L}}^\vee}^\ast[P]$ is called \emph{semisimply differentiable} if the derivative $\boldsymbol{\partial}_\mathcal{F} f$ is a semisimple functional for all forests $\mathcal{F}$ over $P$. 
\end{definition}

Notice that a semisimply differentiable functional is semisimple since, for $\mathcal{F}$ a forest of sticks, the derivative with respect to $\mathcal{F}$ is the identity. Let $\oS^\ast[P]$ denote the subspace of ${\textbf{\textsf{L}}^\vee}^\ast[P]$ of semisimply differentiable functionals. We denote by $\oS[P]$ the quotient of $\textbf{\textsf{L}}^\vee[P]$ which is the linear dual space of $\oS^\ast[P]$. We let
\[      
\oS^\ast[I]:= \oS^\ast\big [(I)\big ] 
\qquad \text{and} \qquad      
\oS[I]:= \oS\big [(I)\big ] 
.\]
Given a subset $S\subseteq I$, we have natural isomorphisms $\oS^\ast[^{I\! } S ] \cong \oS^\ast[S]$ and $\oS[^{I\! } S ] \cong \oS[S]$. 

\begin{cor}\label{cor:delay}
Let $P$ be a partition of $I$. Let $f\in \oS^\ast[P]$ be a semisimply differentiable functional, and let $\mathcal{F}$ be a forest over $P$. Then the forest derivative $\boldsymbol{\partial} _\mathcal{F} f$ does not depend upon the layering of $\mathcal{F}$. 
\end{cor}
\begin{proof}
We may calculate $\boldsymbol{\partial} _\mathcal{F} f$ according to \autoref{prop:delayer}. The result then follows from the fact that the operator $\otimes \boldsymbol{\partial} _ \mathcal{F}$ does not depend upon the layering between the trees of $\mathcal{F}$.   
\end{proof}


Let $\textsf{Lie}_I$ be as in \autoref{sec:lunes}, that is $\textsf{Lie}_I$ is the linear category which is the quotient of $\textsf{Lay}_I$ by the Lie axioms of antisymmetry and the Jacobi identity, and identifying forests which differ only by their layerings.

\begin{cor}\label{imortcor}
Let $\mathcal{F}$ also denote the image of $\mathcal{F}\in \textsf{Lay}_I$ in the quotient $\textsf{Lie}_I$. Then
\[   \textsf{Lie}^{\text{op}}_I \to \textsf{Vec}, \qquad            P\mapsto \oS^\ast[P]     ,\quad     \mathcal{F} \mapsto \boldsymbol{\partial}_\mathcal{F}   \] 
is a well-defined linear functor, and
\[   \textsf{Lie}_I \to \textsf{Vec}, \qquad             P\mapsto \oS[P]     ,\quad    \mathcal{F} \mapsto \boldsymbol{\partial}^\ast_\mathcal{F}   \] 
is a well-defined linear functor.
\end{cor}

The functor $\mathcal{F} \mapsto \boldsymbol{\partial}_\mathcal{F}$ provides the structure for a Lie coalgebra in species, and the functor $ \mathcal{F}\mapsto \boldsymbol{\partial}^\ast_\mathcal{F}$ provides the structure for its dual Lie algebra in species (see \autoref{sec:coalg}).

\begin{thm}\label{cor1}
Let $P=(S_1|\dots |S_k)$ be a partition of $I$. Then the restricted external product
\[             \oS^\ast[^{I\! }  S_1]\otimes \dots \otimes  \oS^\ast[^{I\! }  S_k]\to             \oS^\ast[P]  
, \qquad  
\textbf{\textit{f}}\mapsto \textbf{prod}_P(\textbf{\textit{f}})  \]
is a well-defined isomorphism of vector spaces.
\end{thm}
\begin{proof}
The map is well-defined since for any $\textbf{\textit{f}}\in \bigotimes_j \oS^\ast[^{I\! } S_j ]$, the product $\textbf{prod}_P(\textbf{\textit{f}})$ is semisimply differentiable by \autoref{prop:delayer}. We have already seen that $\textbf{prod}_P$ is injective. 

For surjectivity, let $f\in \oS^\ast[P]$ be any nonzero functional. In particular, $f$ is semisimple. Since semisimple functionals are superpositions of external products of simple functionals, we may assume that $f$ is an external product of simple functionals. Thus, let 
\[       
\textbf{\textit{f}} = f_1\otimes \dots \otimes f_k \in   {\textbf{\textsf{L}}^\vee}^\ast[^{I\! }S_1]    \otimes \dots \otimes  {\textbf{\textsf{L}}^\vee}^\ast[^{I\! }S_k] 
\qquad \text{such that} \qquad \textbf{prod}_P(\textbf{\textit{f}})=f 
. \]
Simple does not imply semisimply differentiable, so we still need to show that 
\[
f_m\in \oS^\ast[^{I\! } S_m] \qquad \text{for each}\quad  1\leq m \leq k     
.\] 
Let $\mathcal{T}_m$ be a tree over the block $S_m$ of $P$. We have
\[  
\boldsymbol{\partial} _{\mathcal{T}_m} f=    \underbrace{\boldsymbol{\partial} _{\mathcal{T}_m} \textbf{prod}_P(    \textbf{\textit{f}} )=  \textbf{prod}_P( \otimes \boldsymbol{\partial}_{\mathcal{T}_m}   \textbf{\textit{f}} )}_{\text{\autoref{prop:delayer}}}=    \textbf{prod}_P  (  f_1\otimes \dots \otimes f_{m-1} \otimes \boldsymbol{\partial}_{\mathcal{T}_m } f_m \otimes f_{m+1} \otimes \dots \otimes f_k) 
.\]
Since $f$ is semisimply differentiable, we have that $ \boldsymbol{\partial} _{\mathcal{T}_m} f$ must be semisimple. Towards a contradiction, suppose that $\boldsymbol{\partial}_{\mathcal{T}_m}  f_m$ is not semisimple. Let $^{I\! }P_m$ denote the partition of $I$ which is the completion of the labels of the leaves of $\mathcal{T}_m$ with singletons, so that 
\[
\boldsymbol{\partial}_{\mathcal{T}_m}  f_m\in    {\textbf{\textsf{L}}^\vee}^\ast[^{I\! }P_m]
.\]  
If $\boldsymbol{\partial}_{\mathcal{T}_m}  f_m$ is not semisimple, there exist adjoint faces $\mathtt{X}_{1},\mathtt{X}_{2}\in \textbf{\textsf{L}}^\vee[^{I\! }P_m]$ with 
\[   \textbf{proj}_{ P_m }(\mathtt{X}_{1})= \textbf{proj}_{ P_m }(\mathtt{X}_{2}) \qquad   \text{but} \quad  \ \     \boldsymbol{\partial}_{\mathcal{T}_m}  f_m(\mathtt{X}_{1})\neq \boldsymbol{\partial}_{\mathcal{T}_m}  f_m(\mathtt{X}_{2}).\] 
We have $f_j\neq 0$ for $1\leq j \leq k$, since we assumed that $f$ is nonzero. Therefore there exists a family of adjoint faces 
\[\mathtt{Z}_j\in \textbf{\textsf{L}}^\vee[^{I\! } S_j],\quad  1\leq j \leq k,  \qquad   \text{such that} \qquad    f_j(\mathtt{Z}_j)\neq 0.\] 
Let $Q$ denote the partition of $I$ which is the completion of the labels of the leaves of $\mathcal{T}_m$ with the blocks of $P$. By \autoref{prop:sur}, there exists an adjoint face $\mathtt{Y}_1 \in \textbf{\textsf{L}}^\vee[Q]$ such that
\[     \textbf{proj}_P(\mathtt{Y}_1)= \mathtt{Z}_1\otimes \dots \otimes  \mathtt{Z}_{m-1}\otimes \mathtt{X}_1\otimes  \mathtt{Z}_{m+1} \otimes   \dots \otimes \mathtt{Z}_m  .   \] 
Similarly, there exists an adjoint face $\mathtt{Y}_2 \in \textbf{\textsf{L}}^\vee[Q]$ such that
\[\textbf{proj}_P(   \mathtt{Y}_2)=\mathtt{Z}_1\otimes \dots \otimes  \mathtt{Z}_{m-1}\otimes \mathtt{X}_2\otimes  \mathtt{Z}_{m+1} \otimes   \dots \otimes \mathtt{Z}_m    .            \] 
Then 
\begin{align*}
 \textbf{proj}_Q(\mathtt{Y}_1) =&\mathtt{Z}_1\otimes \dots \otimes  \mathtt{Z}_{m-1}\otimes \textbf{proj}_{P_m}( \mathtt{X}_1)\otimes  \mathtt{Z}_{m+1} \otimes   \dots \otimes \mathtt{Z}_m\\[4pt]
=& \mathtt{Z}_1\otimes \dots \otimes  \mathtt{Z}_{m-1}\otimes \textbf{proj}_{P_m}(\mathtt{X}_2)\otimes  \mathtt{Z}_{m+1} \otimes   \dots \otimes \mathtt{Z}_m=  \textbf{proj}_Q(\mathtt{Y}_2)
\end{align*}
and  
\begin{align*} 
 \boldsymbol{\partial} _{\mathcal{T}_m} f(\mathtt{Y}_1)= &f_1(\mathtt{Z}_1) \dots  f_{m-1}(\mathtt{Z}_{m-1})\cdot  \boldsymbol{\partial}_{\mathcal{T}_m}  f_m(\mathtt{X}_{1}) \cdot f_{m+1} (\mathtt{Z}_{m+1}) \dots  f_k(\mathtt{Z}_k)  \\[4pt]
\neq &f_1(\mathtt{Z}_1) \dots f_{m-1}(\mathtt{Z}_{m-1})\cdot \boldsymbol{\partial}_{\mathcal{T}_m}  f_m(\mathtt{X}_{2}) \cdot f_{m+1} (\mathtt{Z}_{m+1}) \dots  f_k(\mathtt{Z}_k)= \boldsymbol{\partial} _{\mathcal{T}_m} f(\mathtt{Y}_2).
\end{align*}
This contradicts the fact that $ \boldsymbol{\partial} _{\mathcal{T}_m} f$ is semisimple. Therefore $f_m\in \oS^\ast[^{I\! } S_m]$ for all $1\leq m \leq k$, and so 
\[
f\in \textbf{prod}_P\big (  \oS^\ast[^{I\! } S_1]\otimes  \cdots \otimes  \oS^\ast[^{I\! } S_k] \big )
. \qedhere \]
\end{proof}



\begin{cor}
Let $P=(S_1|\dots |S_k)$ be a partition of $I$. Then
\[       
 \oS[P] \to     \oS[ ^{ I\! } S_1 ]\otimes \dots \otimes  \oS[^{ I\! } S_k] 
, \qquad    
\mathtt{D}\mapsto  \textbf{proj}_P(\mathtt{D})  
\]
is well-defined and is an isomorphism.
\end{cor}
\begin{proof}
This is the linear dual of \autoref{cor1}.
\end{proof}



\subsection{The Steinmann Relations} \label{sec:The Steinmann Relations}

\hyperlink{foo}{We} now characterize the subspace of semisimply differentiable functionals 
\[\oS^\ast[I]\hookrightarrow {\textbf{\text{L}}^\vee}^\ast[I]\] 
by describing a set of relations which generate the kernel of its linear dual map $\textbf{L}^\vee[I]\twoheadrightarrow \oS[I]$.


\begin{definition}\label{def:stein}
Let $\mathcal{V}:(I) \leftarrow (S|T)$ be a cut of $I$ into two blocks, and let $\mathtt{X}_1,\mathtt{X}_2\in \textbf{\textsf{L}}^\vee[S|T]$ be Steinmann adjacent adjoint faces. We call a relation of the form
\[\mathtt{X}_1^\mathcal{V} -\mathtt{X}_1^{\overline{\mathcal{V}}} +\mathtt{X}_2^{\overline{\mathcal{V}}} -\mathtt{X}_2^\mathcal{V}=0\]
a \emph{Steinmann relation} over $I$. 
\end{definition}

This coincides with the definition of Steinmann relations in axiomatic quantum field theory, e.g. \cite[p. 827-828]{streater1975outline}, \cite[Section 4.3]{epstein1976general}. Ocneanu discusses the Steinmann relations in \cite[\href{https://youtu.be/mqHXVDwTGlI?t=2715}{Lecture 36, 45:15}]{oc17}. 

For $f\in {\textbf{\text{L}}^\vee}^\ast[I]$, directly from the definitions we see that $\boldsymbol{\partial} _\mathcal{V} f$ is semisimple if and only if
\[      f(   \mathtt{X}_1^\mathcal{V} -\mathtt{X}_1^{\overline{\mathcal{V}}} +\mathtt{X}_2^{\overline{\mathcal{V}}} -\mathtt{X}_2^\mathcal{V})  =   0  \]
for all Steinmann adjacent adjoint faces $\mathtt{X}_1,\mathtt{X}_2\in  \textbf{\textsf{L}}^\vee[S|T]$. Let 
\[    \textbf{Stein}[I]:=  \big \la   \mathtt{X}_1^\mathcal{V} -\mathtt{X}_1^{\overline{\mathcal{V}}} +\mathtt{X}_2^{\overline{\mathcal{V}}} -\mathtt{X}_2^\mathcal{V} \big     \ra \subset   \textbf{L}^\vee[I]   ,         \]
where $\mathcal{V}:(I) \leftarrow (S|T)$ ranges over cuts of $I$, and $\mathtt{X}_1,\mathtt{X}_2\in  \textbf{\textsf{L}}^\vee[S|T]$ range over Steinmann adjacent adjoint faces. Then $f\in  {\textbf{\text{L}}^\vee}^\ast[I]$ has semisimple first derivatives if and only if
\[
f\in  \Hom \Big ( \bigslant{ \textbf{L}^\vee[I] }{  \textbf{Stein}[I]    } \,    ,  \, \Bbbk \Big)    
 .\] 
The following result shows that this is sufficient to conclude that $f$ is semisimply differentiable. 

\begin{thm}\label{main}
Let $f \in {\textbf{\text{L}}^\vee}^\ast[I]$ be a functional on adjoint chambers. Then $f$ is semisimply differentiable if (and only if) the first derivatives of $f$ are semisimple. Thus,
\[\oS[I] =   \bigslant{ \textbf{L}^\vee[I] }{  \textbf{Stein}[I]     }    .\]  
\end{thm}
\begin{proof}
Let us assume that $f \in {\textbf{\text{L}}^\vee}^\ast[I]$ has semisimple first derivatives, i.e. $\boldsymbol{\partial} _\mathcal{V} f$ is semisimple for all cuts $\mathcal{V}$ of $I$. Consider a second derivative of $f$, i.e. a first derivative of some $\boldsymbol{\partial} _\mathcal{V} f$. By antisymmetry, up to sign this second derivative will be of the form 
\[\boldsymbol{\partial} _{ [[S,T ], U  ]  } f, \qquad    \text{for some} \quad  S\sqcup T\sqcup U=I.\] 
Let $Q$ and $P$ be the partitions
\[        Q = (S|T|U) \qquad \text{and}\qquad     P=           (S\sqcup T|U)  .        \]
By linearity, it is enough to consider the case when the first derivative $\boldsymbol{\partial} _{     [ S\sqcup T, U]    } f$ is a product, so let 
\[\boldsymbol{\partial} _{     [ S\sqcup T, U]    } f=\textbf{prod}_P(f_1\otimes f_2).\] 
Then,
\[\boldsymbol{\partial} _{ [[S,T ], U  ]  } f  = \underbrace{\boldsymbol{\partial} _{[S,T]}\ \textbf{prod}_P (f_1 \otimes f_2)  = \textbf{prod}_P(  \boldsymbol{\partial} _{[S,T]} f_1  \otimes f_2 )}_{ \text{\autoref{prop:delayer}} }   .\] 
Let $P_{S|T}$ be the partition of $I$ which is the completion of the blocks $S$ and $T$ with singletons, so that
\[
\boldsymbol{\partial} _{[S,T]} f_1\in  {\textbf{\textsf{L}}^\vee}^\ast[P_{S|T}]    
.\] 
Towards a contradiction, suppose that $\boldsymbol{\partial} _{ [[S,T ], U  ]  } f$ is not semisimple. A product of semisimple functionals is clearly semisimple. Therefore, since $f_2$ is simple, we have that $\boldsymbol{\partial} _{[S,T]} f_1$ is not semisimple. So there exist adjoint faces $\mathtt{X}_1,\mathtt{X}_2\in \textbf{\textsf{L}}^\vee[P_{S|T}]$ with
\[          \textbf{proj}_{P_{S|T}} (\mathtt{X}_1)   =  \textbf{proj}_{P_{S|T}} (\mathtt{X}_2)  \qquad \text{and} \qquad  \boldsymbol{\partial} _{[S,T]} f_1(\mathtt{X}_1) \neq \boldsymbol{\partial} _{[S,T]} f_1(\mathtt{X}_2).            \] 
We must have $f_2\neq 0$, because otherwise $\boldsymbol{\partial} _{ [[S,T ], U  ]  } f=0$, which is trivially semisimple. So let $\mathtt{Z}$ be any adjoint face such that $f_2(\mathtt{Z})\neq 0$. Let $\mathtt{Y}_1,\mathtt{Y}_2\in \textbf{\textsf{L}}^\vee[Q]$ be adjoint faces such that
\[     \textbf{proj}_P(\mathtt{Y}_1)= \mathtt{X}_1 \otimes \mathtt{Z}   \qquad \text{and} \qquad  \textbf{proj}_P(\mathtt{Y}_2)= \mathtt{X}_2 \otimes \mathtt{Z} .    \]  
Then  
\[ \boldsymbol{\partial} _{ [[S,T ], U  ]  } f(\mathtt{Y}_1)=   \boldsymbol{\partial} _{[S,T]} f_1(\mathtt{X}_1) \cdot f_2(\mathtt{Z})\neq  \boldsymbol{\partial} _{[S,T]} f_1(\mathtt{X}_2) \cdot f_2(\mathtt{Z})=\boldsymbol{\partial} _{ [[S,T ], U  ]  } f(\mathtt{Y}_2).\] 
Recall that the derivative satisfies the Jacobi identity, and so
\[               \boldsymbol{\partial} _{ [[S,T ], U  ]  } =-\boldsymbol{\partial} _{ [[U,S ], T  ] } -\boldsymbol{\partial} _{ [[T,U], S  ]  }   .         \]
Therefore, we have
\begin{equation} \label{eqst} \tag{$1$}
 ( -\boldsymbol{\partial} _{ [[U,S ], T  ]  } -\boldsymbol{\partial} _{ [[T,U ], S ]  } )f (\mathtt{Y}_1)  \neq ( -\boldsymbol{\partial} _{ [[U,S ], T  ] } -\boldsymbol{\partial} _{ [[T,U ], S ]  } )f (\mathtt{Y}_2) . 
\end{equation}
However, by the definition of the derivative, for $\boldsymbol{\partial} _{ [[U,S ], T  ]  }$ we have
\begin{equation} \label{eqst2} \tag{$2$}   
\boldsymbol{\partial} _{ [[U,S ], T ]  } f (\mathtt{Y}_1) =     \boldsymbol{\partial} _{ [  U\sqcup S , T  ]  } f (\mathtt{Y}^{[U,S ]}_1) -  \boldsymbol{\partial} _{ [  U\sqcup S , T  ]  } f (\mathtt{Y}^{[S,U ]}_1)     
\end{equation}
and 
\begin{equation} \label{eqst3} \tag{$3$}  
\boldsymbol{\partial} _{ [[U,S ], T  ]  } f (\mathtt{Y}_2) =     \boldsymbol{\partial} _{ [  U\sqcup S , T  ]  } f (\mathtt{Y}^{[U,S ]}_2) -  \boldsymbol{\partial} _{ [  U\sqcup S , T  ]  } f (\mathtt{Y}^{[S,U]}_2).      
\end{equation}
Let $P_{US}$ denote the partition $(U\sqcup S|T)$. In particular, we have $\boldsymbol{\partial} _{ [ U\sqcup S , T  ]  } f\in    {\textbf{\textsf{L}}^\vee}^\ast[P_{US}]$. Notice that
\[      \textbf{proj}_{P_{US}}(  \mathtt{Y}^{[U,S ]}_1)= \textbf{proj}_{P_{US}}  (      \mathtt{Y}^{[U,S ]}_2   ) \qquad \text{and} \qquad   \textbf{proj}_{P_{US}} ( \mathtt{Y}^{[S,U ]}_1)=  \textbf{proj}_{P_{US}}(\mathtt{Y}^{[S,U ]}_2).  \]
Then, since $\boldsymbol{\partial} _{ [ U\sqcup S , T  ]  } f$ is a first derivative of $f$ and so must be semisimple, we have
\[    \boldsymbol{\partial} _{ [ U\sqcup S , T  ]  } f(\mathtt{Y}^{[U,S ]}_1 ) = \boldsymbol{\partial} _{ [ U\sqcup S , T  ]  } f(\mathtt{Y}^{[U,S ]}_2) \qquad \text{and} \qquad    \boldsymbol{\partial} _{ [ U\sqcup S , T  ]  } f(\mathtt{Y}^{[S,U ]}_1 ) = \boldsymbol{\partial} _{ [ U\sqcup S , T  ]  } f(\mathtt{Y}^{[S,U ]}_2) .       \]
Together with (\ref{eqst2}) and (\ref{eqst3}), this implies
\begin{equation} \label{eqst4} \tag{$\neg 1a$}  
\boldsymbol{\partial} _{ [[U,S ], T  ]  } f (\mathtt{Y}_1)= \boldsymbol{\partial} _{ [[U,S ], T  ]  } f (\mathtt{Y}_2).            
\end{equation}
The following similar equality for $\boldsymbol{\partial} _{ [[T,U ], S  ]  } $ is obtained by the same method,
\begin{equation} \label{eqst5} \tag{$\neg 1b$}  
\boldsymbol{\partial} _{ [[T,U ], S  ]  } f (\mathtt{Y}_1)= \boldsymbol{\partial} _{ [[T,U ], S  ]  }  f (\mathtt{Y}_2).    
\end{equation}
Then (\ref{eqst4}) and (\ref{eqst5}) contradict (\ref{eqst}), and so $\boldsymbol{\partial} _{ [[S,T ], U  ]  } f$ must be semisimple. Thus, we have shown that if all the first derivatives of $f$ are semisimple, then all the second derivatives of $f$ are semisimple. The result then follows by induction on the order of the derivative. 
\end{proof}   

In \cite{oc17}, Ocneanu gave an interesting alternative proof of this result for the special case $n\leq 5$, which features an analysis of the structure of adjoint faces in five coordinates. This proof may generalize to all $n$. 

\begin{cor}
Let $P$ be a partition of $I$, and let $f \in {\textbf{\textsf{L}}^\vee}^\ast[P]$ be a functional on adjoint faces. Then $f$ is semisimply differentiable if (and only if) $f$ is semisimple and has semisimple first derivatives. 
\end{cor}
\begin{proof}
This follows from \autoref{cor1} and \autoref{main}.
\end{proof}

 
\section{A Lie Algebra in Species} \label{sec:coalg}

\noindent \hyperlink{foo}{We} now show that the forest derivative of semisimply differentiable functionals has the structure of a comodule of the Lie cooperad, internal to the category of species. Dually, this endows the adjoint braid arrangement modulo the Steinmann relations with the structure of a Lie algebra in species. Throughout this section, we follow \cite[Appendix B]{aguiar2010monoidal} for species and operads. See also \cite[Section 1.8]{operads}, \cite[Section 5.2]{lodaybook}.

\subsection{The Structure Maps}

By \autoref{cor1}, we can restrict the product $\text{prod}_P$ from \autoref{sec:Functions on Shards} to obtain a bijection
\[           \text{prod}_P |_{\oS} :  \oS^\ast[S_1]\otimes \dots \otimes  \oS^\ast[S_k]\to             \oS^\ast[P]  , \qquad    f\mapsto  \text{prod}_P(f) . \]
We may view this as a geometric realization of the abstract tensor product, so that $\text{T}^\vee[I]$ simultaneously interprets $\oS^\ast[I]$ and $ \oS^\ast[S_1]\otimes \dots \otimes  \oS^\ast[S_k]$. Dually, we have the bijection
\[           \text{proj}_P |_{\oS} :    \oS[P] \to  \oS[S_1]\otimes \dots \otimes  \oS[S_k]       , \qquad    f\mapsto  \text{proj}_P(f) . \]
Similarly, we may view the inverse $(\text{proj}_P |_{\oS})^{-1}$ as a geometric realization. Given a forest $\mathcal{F}:P\leftarrow  Q$ where $P=(S_1|\dots |S_{k_1})$ and $Q=(T_1| \dots| T_{k_2})$, we define
\begin{align*}
  \partial_\mathcal{F}:   \bigotimes_j \oS^\ast[S_j]      \to \bigotimes_j \oS^\ast[T_j]&, \qquad       &&\partial_\mathcal{F}=     (\text{prod}_Q |_{\oS})^{-1} \circ   \boldsymbol{\partial}_\mathcal{F} \circ \text{prod}_P |_{\oS}  .  \\ \intertext{This composition says; geometrically realize the abstract tensor product on the adjoint braid arrangement, perform the forest derivative of functionals on adjoint faces, and then undo the geometric realization to get back an abstract tensor product. The linear dual of $\partial_\mathcal{F}$ is given by }
 \partial^\ast_\mathcal{F}:   \bigotimes_j \oS[T_j]      \to \bigotimes_j \oS[S_j]&, \qquad       &&\partial^\ast_\mathcal{F}=   \text{proj}_P |_{\oS} \circ  \boldsymbol{\partial}^\ast_\mathcal{F} \circ   (\text{proj}_Q |_{\oS})^{-1} .\\ \intertext{In particular, if $\mathcal{F}=\mathcal{T}:I\leftarrow  Q$ is a tree, then we obtain} 
  \partial_\mathcal{T}:  \oS^\ast[I]      \to  \oS^\ast[T_1]\otimes \dots \otimes \oS^\ast[T_k]&  ,      &&\partial_\mathcal{T}=     (\text{prod}_Q |_{\oS})^{-1} \circ   \boldsymbol{\partial}_\mathcal{T}   \\ \intertext{ and }
 \partial^\ast_\mathcal{T}:  \oS[T_1]\otimes \dots\otimes  \oS[T_k]   \to \oS[I]&,    &&\partial^\ast_\mathcal{T}=  \boldsymbol{\partial}^\ast_\mathcal{T} \circ   (\text{proj}_Q |_{\oS})^{-1} .
\end{align*}

\subsection{A Lie Algebra in Species} For background on species, see \cite[Chapter 8]{aguiar2010monoidal}. Let $\textsf{S}$ denote the category of finite sets and bijections, and let $\textsf{Vec}$ denote the category of vector spaces over $\Bbbk$. We have the vector species $\Lay$ of layered trees, 
\[\textbf{Lay}: \textsf{S}  \to \textsf{Vec}, \qquad  I\mapsto  \textbf{Lay}[I].  \] 
We also have the vector species $\oS$ of adjoint chambers modulo the Steinmann relations,
\[\oS  :  \textsf{S} \to \textsf{Vec}, \qquad  I\mapsto  \oS[I] . \] 
We let both species be positive by convention, i.e. $\textbf{Lay}[\emptyset]=0$ and $\oS[\emptyset]=0$. We denote the respective dual species by $\textbf{Lay}^\ast$ and $\oS^\ast$. The category of species is equipped with a monoidal product called \emph{composition} or plethysm, denoted by $\circ$. Monoids internal to species, constructed with respect to composition, are operads by another name. Since $\oS^\ast$ is a positive species, the composition of $\textbf{Lay}^\ast$ with $\oS^\ast$ is given by
\[        \textbf{Lay}^\ast \circ \oS^\ast [I]=    \bigoplus_{P} \Big (       {\textbf{Lay}^\ast}[P] \otimes \bigotimes_j \oS^\ast[ S_j ]        \Big ) . \footnote{\ the direct sum is over all partitions $P$ of $I$}    \]
For each tree $\mathcal{T}\in \textbf{Lay} [P]$, let $\mathcal{T}^\ast\in \textbf{Lay}^\ast [P]$ be given by 
\[\mathcal{T}^\ast(\mathcal{T}'):= \delta_{\mathcal{T},\mathcal{T}'} \qquad  \text{for all trees}\quad \mathcal{T}'\in \textbf{Lay} [P]. \] 
For $f\in  \oS^\ast[I]$, let
\[   \gamma_P(  f   ) :=   \sum_{\mathcal{T} \in  \textbf{Lay}[P]    }  \mathcal{T}^\ast \otimes  \partial_\mathcal{T} f. \]
Let
\[        \gamma:     \oS^\ast  \to    \textbf{Lay}^\ast \circ \oS^\ast  ,\qquad     \gamma(f):= \bigoplus_{P} \gamma_P(f).\]
These linear maps are natural, and so define a morphism of species. Let $\Lie$ denote the (positive) Lie operad, represented using layered trees quotiented by the relations of antisymmetry, the Jacobi identity, and delayering. Our representation of $\Lie$ induces an embedding $\textbf{Lie}^\ast  \hookrightarrow                \textbf{Lay}^\ast$, which in turn induces an embedding $\textbf{Lie}^\ast \circ \oS^\ast \hookrightarrow                \textbf{Lay}^\ast \circ \oS^\ast$.

\begin{prop}
The image of $\gamma$ is contained in the image of $\textbf{Lie}^\ast \circ \oS^\ast \hookrightarrow \textbf{Lay}^\ast \circ \oS^\ast$. 
\end{prop}
\begin{proof}
This is a direct consequence of \autoref{thm:lie} and \autoref{cor:delay}.
\end{proof}


By restricting the image of $\gamma$, we obtain
\[        \gamma|_{\Lie}:     \oS^\ast  \to    \textbf{Lie}^\ast \circ \oS^\ast,\qquad    f\mapsto  \bigoplus_{P} \gamma_P(f)    .\]
We then take the dual of $\gamma|_{\Lie}$, to obtain
\[    \gamma|_{\Lie}^\ast:   \Lie \circ \oS\to \oS, \qquad      \mathcal{T}\otimes  \mathtt{D}\mapsto \partial_\mathcal{T}^\ast\,    \mathtt{D}   . \] 

\begin{thm}\label{mod}
The morphism $\gamma^\ast|_{\Lie}$ is a left $\Lie$-module.
\end{thm}
\begin{proof}
The unit of the Lie operad is the stick. The fact that $ \gamma^\ast|_{\Lie}$ is unital then follows from the fact that $\boldsymbol{\partial}_\mathcal{F}$ is the identity when $\mathcal{F}$ is a forest of sticks. The morphism $ \gamma^\ast|_{\Lie}$ is an action since
\[      \gamma^\ast|_{\Lie}\big(   (\mathcal{T} \circ  \mathcal{F}) \otimes   \mathtt{D}\big ) =              d^\ast_{ \mathcal{T} \circ \mathcal{F}  }\,   \mathtt{D}    =d^\ast_{  \mathcal{T}  }(   d^\ast_{  \mathcal{F}  }\,  \mathtt{D}  )             =\gamma^\ast|_{\Lie}\big (  \mathcal{T} \otimes (  d^\ast_{  \mathcal{F}  }\,  \mathtt{D} )\big).  \qedhere    \]
\end{proof}

\begin{cor}
The morphism $\gamma|_{\Lie}$ is a left $\Lie^\ast$-comodule. 
\end{cor}
\begin{proof}
This is the dual of \autoref{mod}. 
\end{proof}

Left $\Lie$-modules in species with respect to composition are equivalent to Lie algebras in species with respect to the Day convolution, denoted $\otimes_{\text{Day}}$ \cite[Appendix  B.5]{aguiar2010monoidal}. The Lie algebra corresponding to $\gamma^\ast|_{\Lie}$ is given by
\[  \partial^\ast:  \oS \otimes_{\text{Day}} \oS\to  \oS     , \qquad      \mathtt{D} \mapsto \partial^\ast_{[S,T]}\,  \mathtt{D}   , \]
where $\mathtt{D}\in \oS[S]\otimes \oS[T]$. Its dual Lie coalgebra has cobracket the discrete differentiation of functionals across special hyperplanes,
\[   \partial: \oS^\ast \to \oS^\ast \otimes_{\text{Day}} \oS^\ast    , \qquad        f\mapsto    \partial_{ [S,T]}  f   .  \] 

\bibliographystyle{alpha}
\bibliography{steinmann}
\end{document}